\documentclass[11pt]{amsart}

\usepackage{amsmath,amssymb,latexsym,soul,cite,mathrsfs}

\usepackage{color,enumitem,graphicx}
\usepackage[colorlinks=true,urlcolor=blue,
citecolor=red,linkcolor=blue,linktocpage,pdfpagelabels,
bookmarksnumbered,bookmarksopen]{hyperref}
\usepackage[english]{babel}

\usepackage[left=2.9cm,right=2.9cm,top=2.8cm,bottom=2.8cm]{geometry}
\usepackage[hyperpageref]{backref}

\usepackage[colorinlistoftodos]{todonotes}
\makeatletter
\providecommand\@dotsep{5}
\def\listtodoname{List of Todos}
\def\listoftodos{\@starttoc{tdo}\listtodoname}
\makeatother

\numberwithin{equation}{section}

\newtheorem{theorem}{Theorem}[section]
\newtheorem{proposition}[theorem]{Proposition}
\newtheorem{lemma}[theorem]{Lemma}
\newtheorem{corollary}[theorem]{Corollary}

\newtheorem{claim}[theorem]{Claim}

\newtheorem{definition}[theorem]{Definition}

\newcommand\R{\mathbb R}

\begin{document}
	
	\title[Existence of multiple solutions for a Schröndiger logarithmic equation]
	{Existence of multiple solutions for a Schröndiger logarithmic equation via Lusternik-Schnirelmann category}

	\author{Claudianor O. Alves$^*$}
	\author{Ismael S. da Silva}
	
	\address[Claudianor O. Alves]{\newline\indent Unidade Acad\^emica de Matem\'atica
		\newline\indent 
		Universidade Federal de Campina Grande,
		\newline\indent
		58429-970, Campina Grande - PB - Brazil}
	\email{\href{mailto:coalves@mat.ufcg.edu.br}{coalves@mat.ufcg.edu.br}}
	
	\address[Ismael S. da Silva]
	{\newline\indent Unidade Acad\^emica de Matem\'atica
		\newline\indent 
		Universidade Federal de Campina Grande,
		\newline\indent
		58429-970, Campina Grande - PB - Brazil}
	\email{\href{ismael.music3@gmail.com}{ismael.music3@gmail.com}}

	\pretolerance10000
	
	
	\begin{abstract}
		\noindent This paper concerns the existence of multiple solutions for a Schrödinger logarithmic equation of the form
		\begin{equation}
		\left\{\begin{aligned}
		-\varepsilon^2\Delta u + V(x)u & =u\log u^2,\;\;\mbox{in}\;\;\mathbb{R}^{N},\nonumber \\
		u \in H^{1}(\mathbb{R}^{N}),
		\end{aligned}
		\right.\leqno{(P_\varepsilon)}
		\end{equation} 
	where $V:\mathbb{R}^N\longrightarrow \mathbb{R}$ is a continuous function that satisfies some technical conditions and $\varepsilon$ is a positive parameter. We will establish the multiplicity of solution for $(P_\varepsilon)$ by using the notion of Lusternik-Schnirelmann category, by introducing a new function space where the energy functional is $C^1$.  
\end{abstract}
	
	\thanks{Claudianor Alves is the corresponding author and he was partially supported by  CNPq/Brazil 307045/2021-8 and Projeto Universal FAPESQ 3031/2021, e-mail:coalves@mat.ufcg.edu.br}
	\thanks{Ismael da Silva was partially supported by  CAPES, Brazil.}
	\subjclass[2019]{Primary:35J15, 35J20; Secondary: 26A27} 
	\keywords{Schrödinger logarithmic equation, multiple solutions,  Lusternik-Schnirelmann category}

	\maketitle

	\section{Introduction}
	
	In a lot of recent works, the class of elliptic equations of the form
	
	$$ \leqno{(E_1)}	\hspace{4 cm} -\varepsilon^2\Delta u + V(x)u=f(u), \,\,\,\,x \in \mathbb{R}^N,$$
	where $\varepsilon>0$ is a positive parameter and $V$ and $f$ are continuous function under some suitable conditions, has been widly explored for many authors. See, e.g., Alves and Figueiredo \cite{Alves-Giovany}, Ambrosetti, Badiale and Cingolani  \cite{AMbBadCin},  Cingolani and Lazzo \cite{Cingolani-Lazzo}, del Pino and Felmer \cite{del Pino-Felmer}, Oh \cite{Oh1,Oh2}, Rabinowitz \cite{Rabinowitz} and references therein. A motivation for the study of the equation $(E_1)$ is the fact that the solutions of $(E_1)$ have important contributions in some physical problems. It is well known that the solutions of $(E_1)$ are related with the so-called \textit{standing waves solutions} of the nonlinear Schrödinger equation
	$$i\varepsilon\frac{\partial \psi}{\partial t}=-\varepsilon^2 \Delta \psi +(V(x)+K)\psi-f(\psi),$$
	that is, solutions of the form $\psi(x,t)=e^{(-iKt/\varepsilon)}u(x)$, where $u$ is a real value function.  For the reader interesting in Schr\"odinger operators; see e.g. \cite{Almeida,Bergfeldt, Fefferman}.
	
	More recently, many researchers have devoted their attention to the specific case \linebreak $f(t)=t\log t^2$. This special case has relevant physical applications, namely, quantum mechanics, quantum optics, effective quantum gravity, transport and diffusion phenomena and Bose-Einstein condensation (see, e.g., \cite{Zloshchastiev} and references therein). Many authors have presented different approaches in order to find solutions of the equation
	$$ \leqno{(E_2)}	\hspace{4 cm} -\varepsilon^2\Delta u + V(x)u=u\log u^2, \,\,\,\,x \in \mathbb{R}^N,$$
	under different assumptions on $V$ and $\varepsilon$.
	
	The first paper that we would like to cite is a pioneering article  due to Cazenave \cite{Cazenave}. In that work, the author studied the following logarithmic Schrödinger equation, 
	$$\leqno{(E_3)}\hspace{4 cm}iu_t+\Delta u + u\log u^2=0,\,\,\,\,(t,x) \in \mathbb{R}\times\mathbb{R}^N,$$
	by working on the space $W:=\left\{u \in H^1(\mathbb{R}^N);\,\displaystyle{\int_{\mathbb{R}^N}}\mid u^2\log u^2\mid dx<\infty\right\}$ and considering a suitable Luxemburg norm type on $W$. More precisely, by taking the N-function
	$$
	A(s):=\left\{\begin{aligned}
	-\frac{1}{2} s^2 \log s^2,\quad \quad0\leq s&\leq e^{-3}\\
	3s^2 +4e^{-3}s-e^{-6},\quad \quad&s\geq e^{-3},
	\end{aligned}
	\right.
	$$
	$$
	||u||_{A} = \inf \left\{\lambda > 0 \; ; \; \int_{\Omega}A\left(\dfrac{|u|}{\lambda}\right) \leq 1 \right\}
	$$
	and $(||\cdot||_{H^1(\mathbb{R}^N)}+||\cdot||_{A})$ as the norm on $W$, Cazenave showed the existence of infinitely many critical points for the functional 
	$$L(u)=\frac{1}{2}\int_{\mathbb{R}^{N}}|\nabla u|^2dx-\frac{1}{2}\int_{\mathbb{R}^{N}}u^2\log u^2,\quad u \in W$$
	on the set
	$\Sigma:=\left\{u\in W;\,\displaystyle{\int_{\mathbb{R}^N}}|u|^2dx=1\right\}$. Moreover, the properties of the functional 
	$L$ constrained to $\Sigma$ are explored in \cite[Section 4]{Cazenave} to look by solutions of the equation $(E_3)$ given by $\varphi(t,x):=e^{-iHt}\phi(x)$, with $H:=\displaystyle{\inf_{u\in\Sigma}}\, L(u)$ and $E(\phi)=H$. 
	
	Later, others frameworks have been developed in order to find solutions of $(E_2)$. In \cite{d'Avenia}, d'Avenia, Montefusco and Squassina considered the case with $\varepsilon=1$ and $V\equiv 1$, i.e.,
	 \begin{equation}
	 \left\{\begin{aligned}
	 -\Delta u + u  =u\log &u^2,\;\;\mbox{in}\;\;\mathbb{R}^{N},\nonumber \\
	 u \in H^{1}(\mathbb{R}^{N}).
	 \end{aligned}
	 \right.\leqno{(P_1)}
	 \end{equation} 
	In that work the authors have given information about the existence, uniqueness and multiplicity of solutions for $(P_1)$ by employing the critical point theory for nonsmooth functionals developed by Degiovanni and Zani in \cite{Degiovanni}.
	
	Posteriorly, in \cite{Squassina-Szulkin}, Squassina and Szulkin proved the existence of multiple solutions for the problem
	\begin{equation}
	\left\{\begin{aligned}
	-\Delta u &+ V(x)u  =Q(x)u\log u^2,\;\;\mbox{in}\;\;\mathbb{R}^{N},\nonumber \\
	&u \in H^{1}(\mathbb{R}^{N}),
	\end{aligned}
	\right.\leqno{(P_2)}
	\end{equation} 
	where $V$, $Q\in C(\mathbb{R}^N,\mathbb{R})$ are $1$-periodic functions such that
	$$\min_{x \in \mathbb{R}^N}\,Q(x),\,\,\,\min_{x\in\mathbb{R}^N}(V(x)+Q(x))>0.$$
	In that paper, it used the critical point theory for functionals that are sum of a $C^1$-functional with a convex and lower semicontinuous functional due to Szulkin (see \cite{Szulkin}). The theory developed by Szulkin in \cite{Szulkin} has been used in many recent works to study the existence of solution for elliptic equations involving the logarithmic nonlinearity $f(t)=t\log t^2$. In a recent paper of 2018, using a nonsmooth version of the Mountain Pass Theorem found in \cite{Szulkin}, Alves and de Morais Filho showed in \cite{Alves-de Morais} the existence and concentration of positive solutions for the following problem
	\begin{equation}
	\left\{\begin{aligned}
	-\varepsilon^2\Delta u &+ V(x)u  =u\log u^2,\;\;\mbox{in}\;\;\mathbb{R}^{N},\nonumber \\
	&u \in H^{1}(\mathbb{R}^{N}),
	\end{aligned}
	\right.\leqno{(P_3)}
	\end{equation} 
	with $\varepsilon>0$, $N\geq 3$ and $V\in C(\mathbb{R}^N,\mathbb{R})$ verifying 
	$$-1<V_0:=\inf_{x \in \mathbb{R}^N} \,V(x)<\lim_{|x| \to \infty}\,V(x)=:V_\infty<\infty.$$
	Later, Alves and Ji in \cite{Alves-Ji} improved the main  result in \cite{Alves-de Morais} by adapting the penalization method introduced by del Pino and Felmer in \cite{del Pino-Felmer}. In that paper,  the authors established the existence and concentration of positive solutions for $(P_3)$ by assuming that $V\in C(\mathbb{R}^N,\mathbb{R})$ satisfies
	$$V_0=\inf_{x \in \Lambda} V(x) < \min_{x \in \partial \Lambda}V(x),$$
	for some bounded and open set $\Lambda\in \mathbb{R}^N$, with $V_0=\displaystyle \inf_{\R^N} V$ (in \cite{JiXue} and \cite{Xiaoming}  the reader also can find a study on a fractional logarithmic Schrödinger with similar conditions on the potential $V$). In \cite{Alves-Ji}, the authors also used the critical point theory for lower semicontinuous (l.s.c.) functionals found in \cite{Szulkin}. Here, we also cite the works \cite{Alves-Ji2, Ji-Szulkin} in which the theory developed in \cite{Szulkin} was used to study the existence of solution for  a class of logarithmic elliptic equations.

	Motivated by the above mentioned works, in the present paper we are interested in the following problem
	
	\begin{equation}
	\left\{\begin{aligned}
	-\varepsilon^2 \Delta u + V( x)u & =u\log u^2,\;\;\mbox{in}\;\;\mathbb{R}^{N},\nonumber \\
	u \in H^{1}(\mathbb{R}^{N}),
	\end{aligned}
	\right.\leqno{(P_\varepsilon)}
	\end{equation} 
	where $V:\mathbb{R}^N\longrightarrow \mathbb{R}$ is a continuous function satisfying
	\\
	$(V_1)$: $-1<\displaystyle{\inf_{x \in \mathbb{R}^N}}V(x)$;\\
	$(V_2)$: There exists an open and bounded set $\Lambda \subset \mathbb{R}^N$ satisfying
	$$ V_0:=\inf_{x \in \Lambda} V(x) < \min_{x \in \partial \Lambda}V(x).$$ 
	
	Without lost of generality, we will assume throughout this work that $ 0 \in \Lambda$ and $V_0=V(0)$. Note that, by the change of variable $u(x)=v(x/\varepsilon)$, the problem $(P_\varepsilon)$ is equivalent to the problem 
		\begin{equation}
		\left\{\begin{aligned}
			- \Delta v + V(\varepsilon x)v & =v\log v^2,\;\;\mbox{in}\;\;\mathbb{R}^{N},\nonumber \\
			v \in H^{1}(\mathbb{R}^{N}),
		\end{aligned}
		\right.\leqno{(S_\varepsilon)}
	\end{equation}
	
	The problem $(S_\varepsilon)$ has some interesting difficulties in the mathematical point of view. For example, if one tries to apply variational method in order to solve the problem $(S_\varepsilon)$, the natural candidate to be an energy functional is 
	$$E_\varepsilon(u)=\frac{1}{2}\int_{\mathbb{R}^{N}}(|\nabla u|^2+(V(\varepsilon x))|u|^2)dx-\int_{\mathbb{R}^{N}}F(u)dx,$$
	with
	$$F(t)= \displaystyle{\int_{0}^{t}s\log s^2 \,ds}= \frac{1}{2}t^2 \log t^2-\frac{t^2}{2}.$$
	However, it is well known that the functional $E_\varepsilon$ is not well defined on $H^1(\mathbb{R}^N)$ because there exist functions $u \in H^1(\mathbb{R}^N)$ such that $\displaystyle{\int_{\mathbb{R}^N}} u^2\log u^2 =-\infty$, which gives the possibility that $E_\varepsilon(u)=\infty$. In order to carry out this difficulty, in the above mentioned works \cite{Alves-de Morais,Alves-Ji, Ji-Szulkin,Squassina-Szulkin}, the authors have used a decomposition of the form
	\begin{equation}\label{55}
	F_2(t)-F_1(t)=\frac{1}{2}t^2 \log t^2\,\,\,\, \forall t \in \mathbb{R},
	\end{equation}
	where  $F_1$ is a $C^1$ and nonnegative convex function, and $F_2$ is also a $C^1$ function with subcritical growth (see Section 3 in the sequel). Thereby, it is possible to rewrite the functional $E_\varepsilon$ as 
	\begin{equation} \label{NOVAE}
	E_\varepsilon(u)=\frac{1}{2}\int_{\mathbb{R}^{N}}(|\nabla u|^2+(V(\varepsilon x)+1)|u|^2)dx+\int_{\mathbb{R}^N}F_1(u)dx-\int_{\mathbb{R}^{N}}F_2(u)dx.
	\end{equation}
	This little trick ensures that $E_\varepsilon$ can be decomposed as a sum of a $C^1$-functional with a convex and l.s.c. functional. In this way, the critical point theory for l.s.c. functionals developed in \cite{Szulkin} can be used to look by solutions of the problem $(P_\varepsilon)$. 
	
	The variational methods used in \cite{Alves-de Morais,Alves-Ji, d'Avenia, Ji-Szulkin,Squassina-Szulkin} are very interesting, because they permit to work directly with the $H^{1}$-topology or with a similar one, however as in the papers above the energy functional is not $C^1$, some questions involving for example the existence of multiple solutions via the Lusternik-Schnirelmann category cannot be explored in those works. Motivated by this fact, the present intend to treat this situation, by coming back to the approach found in \cite{Cazenave}, but introducing a new Banach space where the energy functional is $C^1$. Hereafter, we will prove that the function $F_1$ in (\ref{55}) (see also (\ref{F1})) is a N-function that satisfies the so-called $(\Delta_2)$-condition. This fact  will permit to consider the reflexive and separable Orlicz space
	$$
	L^{F_1}(\mathbb{R}^N)=\left\{u \in L^{1}_{loc}(\mathbb{R}^N) \; ; \; \int_{\mathbb{R}^N} F_1\left(|u|\right)dx < +\infty \right\}.
	$$
	By setting 
	$$H_\varepsilon:=\left\{u \in H^1(\mathbb{R}^N);\,\int_{\mathbb{R}^{N}}V(\varepsilon x)|u|^2dx<\infty\right\}$$
	and
	$$X_\varepsilon:=H_\varepsilon\cap L^{F_1}(\mathbb{R}^N),$$
	we will prove that $E_\varepsilon$ constrained to the space $X_\varepsilon$ is a $C^1$-functional with
	$$E'_\varepsilon(u)v=\frac{1}{2}\int_{\mathbb{R}^{N}}(\nabla u\nabla v+(V(\varepsilon x)+1)uv)dx+\int_{\mathbb{R}^{N}}F'_1(u)vdx-\int_{\mathbb{R}^{N}}F'_2(u)vdx,\,\,\,\, \forall v \in X_\varepsilon.$$
	Therefore,  using the inequality
	$$\mid t\log t^2\mid\,\leq C(1+\mid t\mid^p),\,\,\,p\in(2,2^*)$$
	and standard arguments of regularity theory,
	a critical point of $E_\varepsilon$ in $X_\varepsilon$ is a classical solution of $(S_\varepsilon)$ (see, e.g., \cite[Section 1]{Alves-de Morais}). 
	
	Using this new approach, we will prove the existence and multiplicity of positive solutions for $(P_\varepsilon)$ by using the  Lusternik-Schnirelmann category theory that is a novelty for this class of problem. Inspired in the ideas presented in Alves and Ji \cite{Alves-Ji}, we also use the penalization method developed by del Pino and Felmer in \cite{del Pino-Felmer}. We would like to point out that our result improve the main result found in \cite{Alves-Ji}, because we are not assuming that $V_0$ in the condition $(V_2)$ above is the global infimum of $V$. Furthermore, we also prove that a concentration phenomena of positive solutions for $(P_\varepsilon)$, in the same line of \cite{Alves-Ji}, it holds for our problem. Adapting some ideas by Alves and Figueiredo \cite{Alves-Giovany} and Cingolani and Lazzo \cite{Cingolani-Lazzo}, we will establish the existence of multiple solutions for $(P_\varepsilon)$ by estimating the number of solutions with the Lusternik-Schnirelmann category  of the set
	$$
	M:=\{x \in \Lambda; \,V(x)=V_0\}
	$$
	in the set
	$$
	M_\delta:=\{x \in \mathbb{R}^N;\,d(x,M)\leq \delta\},
	$$
	for $\delta$ small enough.  
	
	It is very important to mention that in our approach it is crucial that the energy functional be is a $C^1$-functional, because the abstract result involving the Lusternik-Schnirelmann category that we have used in this paper deals with $C^1$-functionals (see Theorem \ref{46} below). Finally, we would like to point out that the ideas developed in the present work could be adapted for the studies found in \cite{Alves-de Morais, Alves-Ji}. Actually, our problem generalizes the cases studied in those works. The reader also can note that our approach gives some aditional information which cannot be found in \cite{Alves-de Morais, Alves-Ji}. The ideas introduced in Section 5 below, for example, can be adapted to show that the versions of Nehari set presented in \cite{Alves-Ji} and \cite{Alves-Ji2} are, under a suitable topology, a $C^1$-manifold. Note also that, as a natural consequence of our results, we get results of multiplicity involving the Lusternik-Schnirelmann category for the problems studied in \cite{Alves-de Morais, Alves-Ji}.
	
 	Our main result in this paper is the following 
 
	\begin{theorem} \label{TherFinal}
	If the conditions $(V_1)-(V_2)$ hold and $\delta>0$ is small enough, then there is $\varepsilon_3>0$,  such that, for $\varepsilon \in (0,\varepsilon_3)$, the following items are valid:
	\begin{itemize}
		\item[$i)$] $(P_\varepsilon)$ has at least $\frac{\text{cat}_{M_\delta}(M)}{2}$ positive solutions , if $\text{cat}_{M_\delta}(M)$ is an even number;
		\item[$ii)$] $(P_\varepsilon)$ has at least $\frac{\text{cat}_{M_\delta}(M)+1}{2}$ positive solutions, if $\text{cat}_{M_\delta}(M)$ is an odd number.  
	\end{itemize}
	\end{theorem}
	
	In the papers \cite{Alves-Ji3,d'Avenia,Ji-Szulkin,Squassina-Szulkin} the reader can find results about the multiplicity of solution for some classes of logarithmic Schrödinger equations, however in those works, the existence of multiple solutions is not associated with the Lusternik-Schnirelmann category.

	The paper is organized as follows: In Section 2, we present a brief review about Orlicz spaces. In Section 3, we introduce an auxiliary problem and the main tools for the variational framework. In Section 4, we prove the existence of solution for the auxiliary problem. In Section 5, we study some properties related with the Nehari set associated with the auxiliary problem and we prove the existence of positive solution for $(S_\varepsilon)$. Finally, in Section 6, we prove our results of multiplicity involving the Lusternik-Schnirelmann category.
	\\ \\
	\textbf{Notation} From now on this paper, otherwise mentioned, we fix:
	\begin{enumerate}
		\item[$\bullet$] $\|\cdot\|_p$ denotes the usual norm of the Lebesgue space $L^p(\mathbb{R}^N)$, $p \in [1, \infty];$
		\item[$\bullet$] If $f:\Omega\rightarrow \mathbb{R}$ is a measurable function, with $\Omega \subset \mathbb{R}^N$ a measurable set, then $\displaystyle{\int_{\Omega}} f(x)\,dx$ will be denoted by $\displaystyle{\int_{\Omega}}f\,;$
		\item[$\bullet$] $o_n(1)$ denotes a real sequence with $o_n(1) \rightarrow 0;$
		\item[$\bullet$] $C(x_1,...,x_n)$ denotes a positive constant that depends of $x_1,...,x_n;$
		\item[$\bullet$] $2^*:= \displaystyle\frac{2N}{N-2}$, if $N\geq 3$ and $2^*:=\infty$ if $N=1$ or $N=2$.	
	\end{enumerate}

	\section{A short review about Orlicz spaces}
	
	In this section, we present some notions and properties related to the Orlicz spaces. For further details see \cite{Adams1, Fukagai 1, RAO}. We start by the following definition:
	
	 \begin{definition}\label{N}
	 	A continuous function $\Phi: \mathbb{R}\rightarrow [0, +\infty)$ is a N-function if:
	 	
	 		\item [(i)] $\Phi$ is convex.
	 		\item [(ii)] $\Phi(t)=0\Leftrightarrow t=0$.
	 		\item [(iii)] $\displaystyle \lim_{t\rightarrow 0} \frac{\Phi(t)}{t}=0$ and  $\displaystyle\lim_{t\rightarrow \infty} \frac{\Phi(t)}{t}=+\infty$.
	 		\item [(iv)] $\Phi$ is an even function.
	 	
	 \end{definition}
	We say that a N-function $\Phi$ verifies the $\Delta_{2}$-condition, denoted by $\Phi \in (\Delta_{2})$, if
	\begin{equation*}
	\Phi(2t) \leq k\Phi(t),\;\;\forall\; t\geq 0,
	\end{equation*}
	for some constant $k>0$. 
	
	The conjugate function $\tilde{\Phi}$ associated with $\Phi$ is given by the Legendre's transformation, more precisely,
	$$\tilde{\Phi}= \max_{t \geq 0} \{st-\Phi(t)\}\;\;\mbox{for}\;\; s \geq 0.$$
	It is possible to prove that $\tilde{\Phi}$ is also a N-function. The functions $\Phi$ and $\tilde{\Phi}$ are complementary to each other, that is, $\tilde{\tilde{\Phi}}=\Phi$.
	
	Given an open set $\Omega \subset \mathbb{R}^N$, we define the Orlicz space associated with the N-function $\Phi$ as
	$$
	L^{\Phi}(\Omega) = \left\{u \in L^{1}_{loc}(\Omega) \; ; \; \int_{\Omega} \Phi\left(\dfrac{|u|}{\lambda}\right) < +\infty, \quad \text{for some} \, \, \lambda >0 \right\}.
	$$
	The space $L^{\Phi}(\Omega)$ is a Banach space endowed with Luxemburg norm given by
	$$
	||u||_{\Phi} = \inf \left\{\lambda > 0 \; ; \; \int_{\Omega}\Phi\left(\dfrac{|u|}{\lambda}\right) \le 1 \right\}.
	$$
	
In the Orlicz spaces, we also have H\"older and Young type inequalities, namely
	$$st\leq \Phi(t)+\tilde{\Phi}(s),\,\,\,\forall s,t\geq 0,$$
	and
	$$
	\left\lvert\int_{\Omega} u v \right\lvert \le 2||u||_{\Phi}||v||_{\tilde{\Phi}},\;\;\forall\; u  \in L^{\Phi}(\Omega) \quad \mbox{and} \quad u  \in L^{\tilde{\Phi}}(\Omega).
	$$
	When $\Phi$, $\tilde{\Phi}\in (\Delta_2)$, then the space $L^{\Phi}(\Omega)$ is reflexive and separable. Furthermore, the \linebreak $\Delta_{2}$-condition implies that
	$$
	L^{\Phi}(\Omega)=\left\{u \in L^{1}_{loc}(\Omega) \; ; \; \int_{\Omega} \Phi\left(|u|\right) < +\infty \right\}
	$$
	and
	$$
	u_{n}\rightarrow u \;\;\mbox{in}\;\; L^{\Phi}(\Omega)\Leftrightarrow \int_{\Omega} \Phi(|u_{n}-u|)\rightarrow 0.
	$$
	
	 We would like to mention an important relation involving N- functions, which will be used later on. Let $\Phi$ be a N-function of $C^1$ class and $\tilde{\Phi}$ its conjugate function. Suppose that
	
\begin{equation}\label{FIN}
	1<l\leq\frac{\Phi'(t)t}{\Phi(t)}\leq m<\infty,\,\,\,\, t\neq 0,
\end{equation}
	then $\Phi$, $\tilde{\Phi} \in (\Delta_2)$. 
	
	Finally, consider
	$$\xi_0(t):=\min\{t^l,\,t^m\}\,\,\, \text{and}\,\,\, \xi_1(t):\max\{t^l,\,t^m\},\,\,\,\,\quad t\geq0.$$
	It is well known that under the condition (\ref{FIN}) the function $\Phi$ satisfies
	\begin{equation}\label{In}
	\xi_0(||u||_{\Phi})\leq \int_{\mathbb{R}^{N}}\Phi(|u|) \leq \xi_1(||u||_\Phi), \quad \forall u \in L^{\Phi}(\Omega).
	\end{equation}

	\section{Variational framework on the logarithmic equation}
	In this section we present the main tools requested to our variational approach. We begin by presenting a suitable decomposition of the nonlinearity $f(t)=t\log t^2$, which is an important step in order to overcome the lack of smoothness of energy functional associated with $(S_\varepsilon)$. Finally, taking into account the conditions $(V_1)-(V_2)$ mentioned above and motivated by \cite{Alves-Ji, del Pino-Felmer}, we introduce an auxiliary problem that is a crucial tool in our study to obtain the existence of solution for $(S_\varepsilon)$.
	\subsection{Basics on the logarithmic equation}
	Let us start by presenting a convenient decomposition of the function
	 $$F(t)= \displaystyle{\int_{0}^{t}s\log s^2 \,ds}= \frac{1}{2}t^2 \log t^2-\frac{t^2}{2},$$ 
	which has been explored in a lot of works (see, e.g., \cite{Alves-de Morais, Alves-Ji, Alves-Ji2, Ji-Szulkin, Squassina-Szulkin}).
	
	Fixed $\delta>0$ sufficiently small, we set
	\begin{equation} \label{F1}
	F_1(s):=\left\{\begin{aligned}
	&0,  \quad \,& s=0\\
	-\frac{1}{2} &s^2 \log s^2,\quad &0<|s|<\delta\\
	-\frac{1}{2} &s^2 (\log \delta^2 +3) + 2\delta|s| - \frac{\delta^2}{2},&|s|\geq \delta
	\end{aligned}
	\right.
	\end{equation}
	and
	$$
	F_2(s):=	\left\{\begin{aligned}
	&0,  \quad \,& |s|<\delta\\
	\frac{1}{2} &s^2 \log \left(\frac{s^2}{\delta^2}\right) + 2\delta|s| -\frac{3}{2}s^2-\frac{\delta^2}{2},\,&|s|\geq \delta
	\end{aligned}
	\right.
	$$
	for every $s\in \R.$ Hence, 
	
	\begin{equation}\label{3}
	F_2(s)-F_1(s)=\frac{1}{2}s^2 \log s^2,\,\,\,\, \forall s \in \mathbb{R}.
	\end{equation}
	It is well known that $F_1$ and $F_2$ verify the properties $(P_1)-(P_4)$ below:
	\\ 
	\begin{itemize}
		\item[($P_1)$] $F_1$ is an even function with $F_1'(s)s\geq 0$ and $F_1\geq 0$. Moreover $F_1 \in C^1(\mathbb{R},\mathbb{R})$ and it is also convex if $\delta \approx 0^+$.
		\item[($P_2)$] $F_2 \in C^1(\mathbb{R},\mathbb{R}) \cap C^{2}((\delta,+\infty),\R)$ and for each $p \in (2,2^*)$, there exists $C=C_p>0$ such that
		$$ |F_2'(s)|\leq C|s|^{p-1},\,\,\,\, \forall s \in \mathbb{R}.$$
		\item[$(P_3)$] $s\mapsto\frac{F'_2(s)}{s}$ is a nondecreasing function for $s>0$ and a strictly increasing function for $s>\delta$.
		\item[$(P_4)$]$\displaystyle{\lim_{s\rightarrow \infty}}\frac{F'_2(s)}{s}=\infty$.
	\end{itemize}

	An important result involving the function $F_1$  in our study is the following
	\begin{proposition}\label{4}
		The function $F_1$ is a N-function. Furthermore, it holds that $F_1$, $\tilde{F_1} \in (\Delta_2)$.
	\end{proposition}
	\begin{proof}
		A direct computation shows that $F_1$ verifies $i)-iv)$ of the Definition \ref{N}. Now, in order to finish the proof we will show that $F_1$ satisfies the relation (\ref{FIN}). First of all, notice that
		$$
		F'_1(s):=\left\{\begin{aligned}
			&-(\log s^2+1)s,\quad &0<s<\delta,\\
			&-s(\log \delta^2 +3) + 2\delta & s\geq \delta.
		\end{aligned}
		\right.
		$$
		Next, we will analyze separately the cases $0<s<\delta$ and $s\geq \delta$.
		\\
		\textbf{Case 1}: $0<s<\delta$.
		
		In this case, 
		\begin{equation*}
		\frac{F'_1(s)s}{F_1(s)}=2+\frac{1}{\log s},
		\end{equation*}
		which implies the existence of $l_1>1$ satisfying 
		\begin{equation}\label{1}
		1<l_1\leq\frac{F'_1(s)s}{F_1(s)}\leq m_1:=\sup_{0<s<\delta}\left(2+\frac{1}{\log s}\right)\leq 2,
		\end{equation}
		for $\delta$ small enough.
		\\
		\textbf{Case 2}: $s\geq\delta$.
		
		In this case, 
		\begin{equation*}
		\frac{F'_1(s)s}{F_1(s)}=\frac{-(\log \delta^2 +3)s^2 + 2\delta s}{-\frac{1}{2}(\log \delta^2 +3)s^2 + 2\delta s-\frac{1}{2}\delta^2}.
		\end{equation*}
		From this, 
		$$\sup_{s\geq \delta}\frac{F'_1(s)s}{F_1(s)}\leq  \sup_{s\geq \delta} \left(\frac{-(\log \delta^2 +3)s^2 + 2\delta s+(2\delta s - \delta^2)}{-\frac{1}{2}(\log \delta^2 +3)s^2 + 2\delta s-\frac{1}{2}\delta^2}\right)\leq 2.$$
		Since 
		$$
		\lim_{s \to +\infty}\frac{F'_1(s)s}{F_1(s)}=2 \quad \mbox{and} \quad \frac{F'_1(s)s}{F_1(s)}> 1, \quad \forall s>0,
		$$
		one gets 
		$$
		1<\inf_{s>0}\frac{F'_1(s)s}{F_1(s)}.
		$$
		
The last inequalities ensure the existence of $l \in (1,2)$ such that
\begin{equation} \label{2}
		1<l\leq\frac{F'_1(s)s}{F_1(s)}\leq 2,\,\,\,\, \forall s>0.
\end{equation} 
	As $F_1$ is an even function, the sentence above holds for any $s\neq 0$ and the proof is finished.
	\end{proof}

In the sequel, in order to avoid the points $u \in H^1(\mathbb{R}^N)$ that verify $F_1(u)\notin L^1(\mathbb{R}^N)$, we will restrict the functional $E_\varepsilon$ given in (\ref{NOVAE}) to the space $X_\varepsilon:=H_\varepsilon\cap L^{F_1}(\mathbb{R}^N),$ which will be denoted by $I_\varepsilon$, that is, $I_\varepsilon\equiv E_\varepsilon|_{X_\varepsilon}$. Here, $L^{F_1}(\mathbb{R}^N)$ designates the Orlicz space associated with $F_1$ and let us consider on $X_\varepsilon$ the norm
$$
||\cdot||_\varepsilon:=||\cdot||_{H_\varepsilon}+||\cdot||_{F_1},
$$
where
$$
||u||_{H_\varepsilon}:=\left(\int_{\mathbb{R}^{N}}(|\nabla u|^2+(V(\varepsilon x)+1)|u|^2)\right)^{1/2},\,\,\,\, u \in H_\varepsilon
$$
and
$$
||u||_{F_1} = \inf \left\{\lambda > 0 \; ; \; \int_{\mathbb{R}^N}F_1\left(\dfrac{|u|}{\lambda}\right) \le 1 \right\}.
$$

	In view of the Proposition \ref{4},  $(X_\varepsilon,\,||\cdot||_\varepsilon)$ is a reflexive and separable Banach space. In this way, from the conditions on $F_1$ and  $V$, one has $I_\varepsilon \in C^1(X_\varepsilon, \mathbb{R})$ with
	$$I'_\varepsilon(u)v=\frac{1}{2}\int_{\mathbb{R}^{N}}(\nabla u\nabla v+(V(\varepsilon x)+1)uv+\int_{\mathbb{R}^{N}}F'_1(u)v-\int_{\mathbb{R}^{N}}F'_2(u)v,\,\,\,\, \forall v \in X_\varepsilon.$$
	Note also that the embeddings  $X_\varepsilon\hookrightarrow H^1(\mathbb{R}^N)$ and $X_\varepsilon\hookrightarrow L^{F_1}(\mathbb{R}^N)$ are continuous.
	\subsection{The auxiliary problem}
 From now on, we fix $b_0\approx 0^+$ and $a_0>\delta$ in a such way that $(\displaystyle{\inf_{\R^N}} \,V+1)>2b_0$ and  $\frac{F'_2(a_0)}{a_0}=b_0$. Using these notations, we set
		$$
		\overline{F}'_2(s):=\left\{\begin{aligned}
			&F'_2(s),\quad &0\leq s\leq a_0;\\
			&b_0s & s\geq a_0.
		\end{aligned}
		\right.
		$$
	Now, consider $t_1$, $t_2>0$ with $a_0 \in (t_1,t_2)$ and $h \in C^1([t_1,t_2])$ verifying
	\\
	$(h_1)$: $h(t)\leq \overline{F}'_2(t),\,\,\, t\in [t_1,t_2]$;\\
	$(h_2)$: $h(t_i)=\overline{F}'_2(t_i)$ and $h'(t_i)=\overline{F}''_2(t_i)$, $i\in\{1,2\}$;\\
	$(h_3)$: $\frac{h(t)}{t}$ is a nondecreasing function. 	\\
	Here, we are using the fact that $F_2 \in C^{2}((\delta,+\infty),\R).$

	Define
	$$
	\tilde{F}'_2(s):=\left\{\begin{aligned}
		&\overline{F}'_2(s),\quad & t \notin [t_1,t_2];\\
		&h(t), & t\in [t_1,t_2].
	\end{aligned}
	\right.
	$$
	Denote by $\chi_\Lambda$ the characteristic function of the set $\Lambda$ and let $g_2:\mathbb{R}^N\times[0,\infty)\longrightarrow\mathbb{R}$ given by
	$$g_2(x,t):=\chi_\Lambda(x) F'_2(t)+(1-\chi_\Lambda(x))\tilde{F}'_2(t).$$
	On account that $F'_2$ is an odd function, we can extend the definition of $g_2$ to $\mathbb{R}^N\times\mathbb{R}$ by setting $g_2(x,t)=-g_2(x,-t)$, for each $t\leq 0$ and $x \in \mathbb{R}^N$.
	
	Hereafter, we will study the existence of solution for the following auxiliary problem
		\begin{equation}
		\left\{\begin{aligned}
		-&\Delta u + (V(\varepsilon x)+1)u =g_2(\varepsilon x,u)-F'_1(u),\;\;\mbox{in}\;\;\mathbb{R}^{N},\nonumber \\
		&u \in H^{1}(\mathbb{R}^{N})\cap L^{F_1}(\mathbb{R}^N).
		\end{aligned}
		\right.\leqno{(\tilde{S}_\varepsilon)}
		\end{equation}
Setting
	$$\Lambda_\varepsilon:=\{x \in \mathbb{R}^N;\, \varepsilon x \in \Lambda\},$$
we see that if $u$  is a positive solution of $(\tilde{S}_\varepsilon)$ satisfying 
	\begin{equation}\label{05}
	0<u(x)<t_1,\,\,\,\, \forall x \in (\mathbb{R}^N-\Lambda_\varepsilon),
	\end{equation}
	then $u$ is a solution of $(S_\varepsilon)$. Have this in mind, we will study the existence of positive solutions for $(S_\varepsilon)$ by looking for solutions of $(\tilde{S}_\varepsilon)$ that satisfy (\ref{05}).
	
	From the definition of $g_2$, it is possible to prove the following properties:
	\\
	$(A_1):$ $
	\left\{\begin{aligned}
	& i):\, g_2(x,t)\leq b_0|t|+C|t|^{p-1},\,\,\,\,\,t\geq 0,\,x\in \mathbb{R}^N;\\
	& ii):\, g_2(x,t)\leq F'_2(t), \,\,\,\,\, x \in \mathbb{R}^N;\\
	& iii):\, g_2(x,t)\leq b_0t,\,\,\,\,\,t\geq 0,\,x\in\left(\mathbb{R}^N-\Lambda\right);\\
	& iv):\,  \frac{1}{2}|t|^2+[F_2(t)-\frac{1}{2}F'_2(t)t+\frac{1}{2}G'_2(\varepsilon x,t)t-G_2(\varepsilon x,t)]\geq 0,\,\,\,\forall t \in \mathbb{R},\, x\in\mathbb{R}^N.
	\end{aligned}
	\right.
	$
		
 Associated with $(\tilde{S}_\varepsilon)$ we have the following functional
$$
J_\varepsilon(u):=\frac{1}{2}\int_{\mathbb{R}^{N}}(|\nabla u|^2+(V(\varepsilon x)+1)|u|^2)+\int_{\mathbb{R}^N}F_1(u)-\int_{\mathbb{R}^{N}}G_2(\varepsilon x,u),\,\,\,\forall u \in X_\varepsilon,
$$
where $G_2(x,t):=\displaystyle{\int_{0}^{t}}g_2(x,s)\,ds$. The conditions on $g_2$ ensures that $J_\varepsilon \in C^1(X_\varepsilon,\mathbb{R})$, and thereby, critical points of $J_\varepsilon$ are weak solutions of $(\tilde{S}_\varepsilon)$.

	\section{Existence of solution for the auxiliary problem}
	
	In this section we will establish the existence of solution for $(\tilde{S}_\varepsilon)$. We start by showing that $J_\varepsilon$ satisfies the geometric configuration of the Mountain Pass Theorem (see \cite{AMBRab}).
	
	\begin{lemma}\label{16}
		Given $\varepsilon>0$, the functional $J_\varepsilon$ satisfies
		\begin{itemize}
			\item [$i)$] There exist $r$, $\rho>0$ such that $J_\varepsilon(u)\geq \rho$ for any $u \in X_\varepsilon$, $||u||_\varepsilon=r$. 
			\item [$ii)$] There exits $v\in X_\varepsilon$ with $||v||_\varepsilon>r$ satisfying $J_\varepsilon(v)<0=J_\varepsilon(0)$.
		\end{itemize}
	\end{lemma}
	\begin{proof}
		$i)$: From $(A_1)$, one has that $G_2(\varepsilon x,t)\leq F_2(t)$, and so,
		$$J_\varepsilon(u)\geq \frac{1}{2}||u||^2_{H_\varepsilon}+\int_{\mathbb{R}^{N}}F_1(u)-\int_{\mathbb{R}^{N}}F_2(u).
		$$
	Gathering (\ref{In}) with (\ref{2}) (note that $m$ can be chosen equal to 2) and using $(P_2)$, there is $r\approx0^+$ such that 
		$$
		J_\varepsilon(u)\geq \frac{1}{2}||u||^2_{H_\varepsilon}+||u||_{F_1}^2-D||u||_\varepsilon^p \geq C||u||^2_\varepsilon-D||u||_\varepsilon^p,
		$$
	 for some $C$, $D>0$. The last inequality gives the desired condition, because $p>2$.
		\\
		$ii)$: Fix $u \in O_\varepsilon:=\{u \in X_\varepsilon;\, |\text{supp}(|u|)\cap\Lambda_\varepsilon|>0\}.$ Note that, for each $x\in\mathbb{R}^N$ we can write
		$$
		F_1(t)=\chi_{\Lambda_\varepsilon}(x) F_1(t)+(1-\chi_{\Lambda_\varepsilon}(x))F_1(t).
		$$
		Therefore, from the definition of $g_2$, 
		$$\begin{aligned}
		J_\varepsilon(tu)\leq\frac{t^2}{2}||u||_{H_\varepsilon}^2-\frac{1}{2}\int_{\mathbb{R}^N}\chi_{\Lambda_\varepsilon}|tu|^2\log |tu|^2&+\frac{1}{2}\int_{[t|u|\leq t_1]}(1-\chi_{\Lambda_\varepsilon})|tu|^2\log |tu|^2+\\
		&+\int_{[t|u|> t_1]}(1-\chi_{\Lambda_\varepsilon})[F_1(tu)-\tilde{F}_2(tu)].
		\end{aligned}$$
	Recalling that $X_\varepsilon\hookrightarrow L^{2}(\mathbb{R}^N)$, there is $C>0$ independent of $t$ such that 
		$$ \int_{[t|u|> t_1]}|tu|^2\leq C,$$
		and so, 
		$$|[t|u|> t_1]|\leq \frac{C}{t_1^2}t^2=:C_1t^2.$$
		By the definition of $F_1$, 
		$$F_1(t)\leq At^2+B,\,\,\,\, t\geq 0,$$
		with $A$, $B>0$. Then, 
		$$\int_{[t|u|> t_1]}(1-\chi_{\Lambda_\varepsilon})F_1(t|u|)\leq Dt^2,\,\,\,$$
		for a convenient $D>0$.	Since $\tilde{F}_2\geq0$, we find
		$$
		\begin{aligned}
		J_\varepsilon(tu)\leq t^2[\,\,\frac{1}{2}||u||_{H_\varepsilon}^2-\int_{\mathbb{R}^N}\chi_{\Lambda_\varepsilon}|u|^2\log|u|^2 &-\log t\left(\int_{\mathbb{R}^N}\chi_{\Lambda_\varepsilon}|u|^2+\int_{[t|u|\leq t_1]}(\chi_{\Lambda_\varepsilon}-1)|u|^2\right)\\
		&+\int_{[t|u|\leq t_1]}(1-\chi_{\Lambda_\varepsilon})|u|^2\log|u|^2+D\,\,].
		\end{aligned}$$
		By the Lebesgue Dominated Convergence Theorem, we have
		$$
		\int_{[t|u|\leq t_1]}(\chi_{\Lambda_\varepsilon}-1)|u|^2\longrightarrow 0,\,\,\,\text{as}\,\,\, t\rightarrow +\infty.
		$$
		Note also that, since $u\in O_\varepsilon$ and $u \in L^{F_1}(\mathbb{R})$, it holds
		$$\int_{\mathbb{R}^N}\chi_{\Lambda_\varepsilon}|u|^2>0$$
		and
		$$\frac{1}{2}\int_{[t|u|\leq t_1]}(1-\chi_{\Lambda_\varepsilon})|u|^2\log|u|^2	\leq\int_{\mathbb{R}^{N}}F_2(u)\,dx<\infty.$$
		Combining all of the above information we derive that
		$$J_\varepsilon(tu)\rightarrow -\infty,\,\,\,\text{as}\,\,\, t\rightarrow \infty,$$
		and the proof is finished by taking $v=tu$ with $t$ large enough.	
	\end{proof}

	For the next lemma, we have adapted the reasoning employed in \cite[Lemma 3.1]{Alves-Ji2}. However, taking into account that in our case the functional $J_\varepsilon$ has on $X_\varepsilon$ a different topology of $H^{1}(\mathbb{R}^N)$, it was necessary to develop new estimates that are not found in \cite{Alves-Ji2}.   
	
In the sequel, we will need of the following logarithmic inequality (see \cite[pg 153]{del Pino-Dolbeaut})
	\begin{equation*}
	\int_{\mathbb{R}^{N}}|u|^2\log\left(\frac{|u|}{||u||_2}\right)\leq C||u||_2\log\left(\frac{||u||_{2^*}}{||u||_2}\right),\,\,\,\, \forall u \in L^2(\mathbb{R}^N)\cap L^{2^*}(\mathbb{R}^N),
	\end{equation*}
	for some positive constant $C$. As an immediate consequence,  
\begin{equation}\label{LogIn}
	\int_{\Lambda_\varepsilon}|u|^2\log\left(\frac{|u|}{||u||_{L^2(\Lambda_\varepsilon)}}\right)\leq C||u||_{L^2(\Lambda_\varepsilon)}\log\left(\frac{||u||_{L^{2^*}(\Lambda_\varepsilon)}}{||u||_{L^{2}(\Lambda_\varepsilon)}}\right),\,\,\,\, \forall u \in L^2(\Lambda_\varepsilon)\cap L^{2^*}(\Lambda_\varepsilon).
\end{equation}
	\begin{lemma}\label{9}
		Let $(v_n)$ be a $(PS)_c$ sequence for $J_\varepsilon$. Then, the sequence $(v_n)$ is bounded in $X_\varepsilon$. 
	\end{lemma}
	\begin{proof}
		Let $(v_n)$ be a $(PS)_c$ sequence for $J_\varepsilon$. Then, 
		\begin{equation}\label{5}
		J_\varepsilon(v_n)-\frac{1}{2}J'_\varepsilon(v_n)v_n\leq (c+1)+o_n(1)||v_n||_\varepsilon,
		\end{equation}
		for large $n$.
		
		On the other hand, observe that
		\begin{equation}\label{6}
		\begin{aligned}
		J_\varepsilon(v_n)-\frac{1}{2}J'_\varepsilon(v_n)v_n&=\int_{\mathbb{R}^{N}}(F_1(v_n)-\frac{1}{2}F'_1(v_n)v_n)+\int_{\mathbb{R}^N}(\frac{1}{2}G'_2(\varepsilon x,v_n)v_n-G_2(\varepsilon x,v_n))=\\
		&=\frac{1}{2}\int_{\mathbb{R}^N}|v_n|^2+\int_{\mathbb{R}^{N}}[F_2(v_n)-\frac{1}{2}F'_2(v_n)v_n+\frac{1}{2}G'_2(\varepsilon x,v_n)v_n-G_2(\varepsilon x,v_n)],
		\end{aligned}
		\end{equation}
		because
		$$
		\int_{\mathbb{R}^{N}}[(F_1(v_n)-\frac{1}{2}F'_1(v_n)v_n)+(\frac{1}{2}F'_2(v_n)v_n-F_2(v_n))]=\frac{1}{2}\int_{\mathbb{R}^N}|v_n|^2.
		$$
		Consequently, 
		$$\begin{aligned}
		J_\varepsilon(v_n)-\frac{1}{2}J'_\varepsilon(v_n)v_n&\geq\frac{1}{2}\int_{\Lambda_\varepsilon}|v_n|^2+\int_{\{\Lambda_\varepsilon^c\cap[|v_n|> t_1]\}}(\frac{1}{2}|v_n|^2+F_2(v_n)-\frac{1}{2}F'_2(v_n)v_n)\,\,+\\
		&+\int_{\{\Lambda_\varepsilon^c\cap[|v_n|> t_1]\}}(\frac{1}{2}G'_2(\varepsilon x,v_n)v_n-G_2(\varepsilon x,v_n)).
		\end{aligned}
		$$
		From $(A_1)-iv)$,
		
		$$J_\varepsilon(v_n)-\frac{1}{2}J'_\varepsilon(v_n)v_n\geq\frac{1}{2}\int_{\Lambda_\varepsilon}|v_n|^2$$
		and so, from (\ref{5}),
		\begin{equation}\label{7}
		(c+1)+o_n(1)||v_n||_\varepsilon\geq\frac{1}{2}\int_{\Lambda_\varepsilon}|v_n|^2.
		\end{equation}
		Recall that there are constants $A$, $B>0$ such that
		$$F_1(t)\leq A|t|^2+B, \,\,\,\forall t\in \mathbb{R}.$$
		This together with (\ref{7}) leads to 
		\begin{equation}\label{22}
		\int_{\Lambda_\varepsilon}F_1(v_n)\leq C_\varepsilon +||v_n||_\varepsilon,
		\end{equation}
		for some $C_\varepsilon>0$. Thanks to (\ref{LogIn}), 	
		$$\begin{aligned}
		\frac{1}{2}&\int_{\Lambda_\varepsilon}|v_n|^2\log |v_n|^2\leq C||v_n||_{L^2(\Lambda_\varepsilon)}\log\left(\frac{||v_n||_{L^{2^*}(\Lambda_\varepsilon)}}{||v_n||_{L^2(\Lambda_\varepsilon)}}\right)+||v_n||_{L^2(\Lambda_\varepsilon)}^2\log (||v_n||_{L^2(\Lambda_\varepsilon)})=\\
		&=(||v_n||_{L^2(\Lambda_\varepsilon)}^2-C||v_n||_{L^2(\Lambda_\varepsilon)})\log (||v_n||_{L^2(\Lambda_\varepsilon)})+C||v_n||_{L^2(\Lambda_\varepsilon)}\log\left(||v_n||_{L^{2^*}(\Lambda_\varepsilon)}\right).
		\end{aligned}
		$$
	that combines with the embedding $X_\varepsilon \hookrightarrow H_\varepsilon$ to give
		$$\int_{\Lambda_\varepsilon}|v_n|^2\log |v_n|^2\leq (2||v_n||_{L^2(\Lambda_\varepsilon)}^2-2C||v_n||_{L^2(\Lambda_\varepsilon)})\log (||v_n||_{L^2(\Lambda_\varepsilon)})+\tilde{C}||v_n||_\varepsilon\left|\log(\tilde{C}||v_n||_\varepsilon)\right|,$$
		for some convenient $\tilde{C}>0$ independent of $\varepsilon$. In order to get the last inequality, we have explored the fact that the function $t\mapsto \log t$, $t>0$, is increasing.
Now, using the fact that  given $r\in(0,1)$ there is $A>0$ satisfying  
		$$|t \log t|\leq A(1+|t|^{1+r}),\,\,\, t\geq0,$$
	we obtain, by gathering this inequality with (\ref{7}), the inequalities below  
	$$
	||v_n||_{L^2(\Lambda_\varepsilon)}\log (||v_n||_{L^2(\Lambda_\varepsilon)})\leq A(1+||v_n||_{L^2(\Lambda_\varepsilon)}^{1+r})$$
	and
	$$
	||v_n||_{L^2(\Lambda_\varepsilon)}^2\log (||v_n||_{L^2(\Lambda_\varepsilon)}^2)\leq A(1+(||v_n||_{L^2(\Lambda_\varepsilon)}^2)^{1+r})\leq \tilde{A}(1+||v_n||_{L^2(\Lambda_\varepsilon)})^{1+r}).
	$$
	From these information, modifying $A$ if necessary, we arrive at
		\begin{equation}\label{8}
		\int_{\Lambda_\varepsilon}|v_n|^2\log |v_n|^2\leq A(1+||v_n||_\varepsilon^{1+r}).
		\end{equation}
As $(v_n)$ is a $(PS)_c$ sequence for $J_\varepsilon$, 
		$$(c+1)\geq J_\varepsilon(v_n)=\frac{1}{2}||v_n||_{H_\varepsilon}^2+\int_{\Lambda_\varepsilon^c}F_1(v_n)-\int_{\Lambda_\varepsilon}|v_n|^2\log |v_n|^2-\int_{\Lambda_\varepsilon^c}G_2(\varepsilon x, v_n)$$
for large $n$. 	From $(A_1)$, 
		$$G_2(\varepsilon x, t) \leq \frac{b_0}{2}t^2,\,\,\,\, \forall x \in \Lambda_\varepsilon^c,$$
then 
$$
(c+1)+A(1+||v_n||_\varepsilon^{1+r})\geq C||v_n||_{H_\varepsilon}^2+\int_{\Lambda_\varepsilon^c}F_1(v_n),
$$ 
for some $C>0$,	and so, by (\ref{22}),
		\begin{equation}\label{23}
		D_\varepsilon+||v_n||_\varepsilon+A(1+||v_n||_\varepsilon^{1+r})\geq \tilde{C}\left(||v_n||_{H_\varepsilon}^2+\int_{\mathbb{R}^N}F_1(v_n)\right),
		\end{equation}
		where $D_\varepsilon:=(C_\varepsilon+c+1)>0$ and $\tilde{C}:=\min\{C,1\}$.  From now on in this proof, we fix $r\in (0,1)$ so that $1+r<l$, where $l$ is the number obtained in (\ref{2}). 
		
		Suppose that  $||v_n||_{F_1}\leq1$. Employing (\ref{In}) in (\ref{23}), and modifying $\tilde{C}$ if necessary, one gets
	    \begin{equation}\label{24}
	    D_\varepsilon+||v_n||_\varepsilon+A(1+||v_n||_\varepsilon^{1+r})\geq \tilde{C}(||v_n||_{H_\varepsilon}+||v_n||_{F_1})^2=\tilde{C}||v_n||_\varepsilon^2.
	    \end{equation}
	    Otherwise, if $||v_n||_{F_1}>1$, we have two possibilities: $||v_n||_{H_\varepsilon}>1$ or $||v_n||_{H_\varepsilon}\leq 1$. When $||v_n||_{H_\varepsilon}>1$, in the same way of the preceding case we obtain
	    \begin{equation}\label{25}
	    D_\varepsilon+||v_n||_\varepsilon+A(1+||v_n||_\varepsilon^{1+r})\geq C_l||v_n||_\varepsilon^l.
	    \end{equation} 	
	    If it occurs $||v_n||_{F_1}>1$ and $||v_n||_{H_\varepsilon}\leq 1$, using the definition $||\cdot||_\varepsilon$ in (\ref{23}), we find
	    \begin{equation}\label{26}
	   	\tilde{D}_\varepsilon+||v_n||_{F_1}+C_r||v_n||_{F_1}^{1+r}\geq\tilde{C}||v_n||_{F_1}^l.
	    \end{equation}
	    The proof is completed by combining (\ref{24})-(\ref{26}).    
			
	\end{proof}	
	
	
	Next, we present an important property of the $(PS)$ sequences that is crucial in order to prove that $J_\varepsilon$ satisfies the $(PS)$ condition in the space $X_\varepsilon$.
	
	\begin{lemma}\label{11}
		Let $(v_n)$ be a $(PS)_c$ sequence for $J_\varepsilon$. Then, given $\tau>0$ there is $R>0$ such that
		$$\limsup_{n\rightarrow \infty}\int_{B^c_R(0)}(|\nabla v_n|^2+(V(\varepsilon x)+1)|v_n|^2)<\tau.$$
	\end{lemma} 
	\begin{proof}
		See \cite[Lemma 3.4]{Alves-Ji}.
	\end{proof}	
	\begin{corollary}\label{12}
		The functional $J_\varepsilon$ satisfies the $(PS)$ condition.
	\end{corollary}
	\begin{proof}
		Let $(v_n)$ be a $(PS)_c$ sequence for $J_\varepsilon$. Without loss of generality we may assume that  $v_n\rightharpoonup v$ in $X_\varepsilon$ for some $v\in X_\varepsilon$. Moreover,  arguing as in \cite[Section 2]{Alves-Airton}, we also have $J_\varepsilon'(v)=0$, and so,  $J_\varepsilon'(v)v=0$, i.e.,
		\begin{equation}\label{10}
		||v||_{H_\varepsilon}^2+\int_{\mathbb{R}^{N}}F'_1(v)v=\int_{\mathbb{R}^{N}}G'_2(\varepsilon x, v)v.
		\end{equation}
		
	As the embedding $X_\varepsilon\hookrightarrow L^q(B_R(0))$ is compact for each $R>0$ and $p\in[2,2^*)$, the growth condition on $G'_2$ (see $(A_1)$ ) together with the Lemma \ref{11} yields 
		$$\int_{\mathbb{R}^{N}}G'_2(\varepsilon x, v_n)v_n\longrightarrow \int_{\mathbb{R}^N}G'_2(\varepsilon x, v)v.$$
		Taking into account this information and using the fact that $(v_n)$ is $(PS)$ sequence, we find
		$$||v_n||_{H_\varepsilon}^2+\int_{\mathbb{R}^{N}}F'_1(v_n)v_n=\int_{\mathbb{R}^{N}}G'_2(\varepsilon x, v_n)v_n+o_n(1).
		$$
The last equality combined with (\ref{10}) implies that 
		$$||v_n||_{H_\varepsilon}^2+\int_{\mathbb{R}^{N}}F'_1(v_n)v_n=||v||_{H_\varepsilon}^2+\int_{\mathbb{R}^{N}}F'_1(v)v+o_n(1),
		$$
from where it  follows that 
		\begin{equation}\label{13}
		||v_n||_{H_\varepsilon}^2\rightarrow||v||_{H_\varepsilon}^2
		\end{equation}
		and
		\begin{equation}\label{14}
		\int_{\mathbb{R}^{N}}F'_1(v_n)v_n \longrightarrow \int_{\mathbb{R}^{N}}F'_1(v)v,
		\end{equation}
and so, $v_n\rightarrow v$ in $H_\varepsilon$. It remains to show that $v_n\rightarrow v$ in $L^{F_1}(\mathbb{R}^N)$. Note that, since $F'_1(t)t\geq 0$, the convergence in (\ref{14}) means that 
		$$F'_1(v_n)v_n\rightarrow F'_1(v)v\,\,\,\,\text{in}\,\,\,\,L^1(\mathbb{R}^N).$$
		This fact associated with (\ref{2}) and Lebesgue's Dominated Convergence Theorem shows that, going to a subsequence if necessary, 
		$$F_1(v_n) \rightarrow F_1(v)\,\,\,\,\text{in}\,\,\,\,L^1(\mathbb{R}^N).$$
		Finally, using that $F_1 \in (\Delta_2)$, we deduce that
		$$\int_{\mathbb{R}^{N}}F_1(|v_n-v|)\longrightarrow 0,$$
		showing that $v_n\rightarrow v$ in $L^{F_1}(\mathbb{R}^N)$, which finishes the proof.
	\end{proof}	
	
	The main result of this section reads as follows
	\begin{theorem}\label{53}
For each  $\varepsilon>0$ the functional $J_\varepsilon$ has a nontrivial critical point $u_\varepsilon$. Consequently, $(\tilde{S}_\varepsilon)$ has a nontrivial solution.
\end{theorem}
\begin{proof}
	By Lemma \ref{16} and Corollary \ref{12}, we see that the functional $J_\varepsilon$ satisfies the assumptions of the Mountain Pass Theorem found in \cite[Theorem 2.1]{AMBRab}, then the mountain pass level given by 
	$$c_\varepsilon:=\inf_{\gamma \in \Gamma_\varepsilon}\max_{t \in [0,1]}J_\varepsilon(\gamma(t))$$
	with
	$$\Gamma_\varepsilon:=\{\gamma \in C([0,1],X_\varepsilon);\, \gamma(0)=0\,\,\,\text{and}\,\,\,J_\varepsilon(\gamma(1))<0\},$$
	is a critical point of $J_\varepsilon$. 
\end{proof}

	From now on, otherwise mentioned, the notation $u_\varepsilon$ designates the solution of $(\tilde{S}_\varepsilon)$ given in the preceding theorem.
	
	\section{The Nehari manifold and the existence of positive solution for $(P_\varepsilon)$}	
	
	In this section we will prove that the Nehari set associated with $J_\varepsilon$, namely
	$$\mathcal{N}_\varepsilon:=\{u\in X_\varepsilon-\{0\};\,J'_\varepsilon(u)u=0\},$$
	is a $C^1$-manifold and that critical points of $J_\varepsilon|_{\mathcal{N}_\varepsilon}$ are critical points of $J_\varepsilon$ in the usual sense. Furthermore, by studying the behavior of levels $c_\varepsilon$ as $\varepsilon \rightarrow 0^+$, we will prove some properties related with $\mathcal{N}_\varepsilon$ that allow us to prove that the solutions $u_\varepsilon$ of $(\tilde{S}_\varepsilon)$ are solutions of $(S_\varepsilon)$ for $\varepsilon \approx 0^+$.
	
	\subsection{Main properties of $\mathcal{N}_\varepsilon$} First of all, set 
	$$
	\Psi_\varepsilon(u):=J_\varepsilon(u)-\frac{1}{2}\int_{\mathbb{R}^N}|u|^2-\left[\int_{\mathbb{R}^{N}}[F_2(u)-\frac{1}{2}F'_2(u)u+\frac{1}{2}G'_2(\varepsilon x,u)u-G_2(\varepsilon x,u)]\right].
	$$
	Accordingly to $(\ref{6})$,  
	$$
	\mathcal{N}_\varepsilon = \Psi_\varepsilon^{-1}(\{0\}).
	$$ 
	
	We start our study with the following result
	\begin{proposition}\label{21}
		There exists $\beta>0$, such that 
		$$||u||_\varepsilon \geq ||u||_{H_\varepsilon}\geq\beta,\,\,\,\forall u \in \mathcal{N}_\varepsilon,$$
		for all $\varepsilon>0$.
	\end{proposition}
	\begin{proof}
		For each $u \in \mathcal{N}_\varepsilon$, 
		$$\int_{\mathbb{R}^N}(|\nabla u|^2+(V(\varepsilon x)+1)|u|^2) +\int_{\mathbb{R}^N}F'_1(u)u=\int_{\mathbb{R}^N}G'_2(\varepsilon,u)u.$$
		Therefore, from $(A_1)$,  
		\begin{equation}\label{17}
		\int_{\mathbb{R}^N}(|\nabla u|^2+(\alpha_0+1-b_0)|u|^2)\leq C\int_{\mathbb{R}^N}|u|^p,
		\end{equation}
		where $\alpha_0=\displaystyle{\inf_{\R^N}} V$. The number $b_0$ has been chosen so that $\alpha_0+1-b_0>0$, then the expression 
		$$||u||_0^2:=\int_{\mathbb{R}^N}(|\nabla u|^2+(\alpha_0+1-b_0)|u|^2)$$
		defines a norm on $H^1(\mathbb{R}^N)$. Setting $H=(H^1(\mathbb{R}^N), ||\cdot||_0)$, one sees that the embedding $H\hookrightarrow L^p(\mathbb{R}^N)$ is continuous. From (\ref{17}), 
		$$ M\leq ||u||_0^{p-2},$$
		for a convenient $M>0$ that is independent of $\varepsilon$. The last inequality yields
		$$0<\beta:=M^{\frac{1}{(p-2)}}\leq ||u||_0\leq ||u||_{H_\varepsilon}\leq||u||_\varepsilon.$$
	\end{proof}	
	For the sake of completeness, we would like to mention that repeating the ideas found in \linebreak \cite[Lemma 3.6 and Remark 3.1]{Alves-Ji}, it can be proved the following lemma 
	\begin{lemma} 	\label{27} For each $u \in O_\varepsilon=\{u \in X_\varepsilon;\, |\text{supp}(|u|)\cap\Lambda_\varepsilon|>0\}$, there is a unique $t_u>0$ such that $t_uu \in \mathcal{N}_\varepsilon$. Reciprocally, if $u \in \mathcal{N}_\varepsilon$, then $u \in O_\varepsilon.$
				
	\end{lemma}

	In the next proposition we prove that $\mathcal{N}_\varepsilon$ is a $C^1$-manifold for each $\varepsilon>0$.
	
	\begin{proposition}\label{20}
$\mathcal{N}_\varepsilon$ is a $C^1$-manifold for each $\varepsilon>0$. 
	\end{proposition}
	\begin{proof}
In the sequel we will prove that for all $u \in \mathcal{N}_\varepsilon$ we must have $\Psi'_\varepsilon(u)u \not= 0.$
Assume by contradiction that there is $u \in \mathcal{N}_\varepsilon$ with $\Psi'_\varepsilon(u)u=0$, i.e.,
		$$0=-\int_{\mathbb{R}^N}|u|^2-\left[\int_{\mathbb{R}^N}(\frac{1}{2}F'_2(u)u-\frac{1}{2}F''_2(u)u^2)+\int_{\mathbb{R}^N}(\frac{1}{2}G''_2(\varepsilon x, u)u^2-\frac{1}{2}G'_2(\varepsilon x,u)u)\right].$$
		Using that $G'_2\equiv F'_2$ in $\Lambda_\varepsilon$, we find
		\begin{equation}\label{18}
		0=-\int_{\Lambda_\varepsilon}|u|^2-\left[\int_{\Lambda_\varepsilon^c}(|u|^2+\frac{1}{2}F'_2(u)u-\frac{1}{2}F''_2(u)u^2)+\int_{\Lambda_\varepsilon^c}(\frac{1}{2}G''_2(\varepsilon x, u)u^2-\frac{1}{2}G'_2(\varepsilon x,u)u)\right].
		\end{equation}
By the definition of $F_2$, 
		$$F'_2(s):=	\left\{\begin{aligned}
			&0 , \quad \,& s\in[0,\delta];\\
			 &s \log \left(\frac{s^2}{\delta^2}\right) + 2\delta -2s,\,&|s|\geq \delta,
		\end{aligned}
		\right.
		$$
and so, 
$$
t^2+\frac{1}{2}F'_2(t)t-\frac{1}{2}F''_2(t)t^2=\delta t>0,\,\,\,\, t\geq \delta,
$$
leading to 
$$
|u|^2+\frac{1}{2}F'_2(u)u-\frac{1}{2}F''_2(u)u^2\geq 0,\,\,\,\text{a.e}\,\,\,x \in \Lambda_\varepsilon^c.
$$
Using this information and the fact that $G'_2(\varepsilon x,t)\equiv F'_2(t)$, for $x \in  \Lambda_\varepsilon^c$ and $t\leq t_1$ in (\ref{18}), we arrive at
$$
\int_{\Lambda_\varepsilon}|u|^2\leq-\int_{\Lambda_\varepsilon^c\cap[t_1<|u|<t_2]}(\frac{1}{2}G''_2(\varepsilon x, u)u^2-\frac{1}{2}G'_2(\varepsilon x,u)u)\\
-\int_{\Lambda_\varepsilon^c\cap[|u|\geq t_2]}(\frac{1}{2}G''_2(\varepsilon x, u)u^2-\frac{1}{2}G'_2(\varepsilon x,u)u).
$$	
As $G'_2(\varepsilon x,u)=h(u)$ for $x \in \Lambda_\varepsilon^c$ and $u(x) \in (t_1,t_2)$,  $(h_3)$ gives 
	$$G''_2(\varepsilon x, u)u^2-\frac{1}{2}G'_2(\varepsilon x,u)u=\frac{1}{2}(h'(u)u-h(u))u\geq 0,\,\,\, \text{a.e}\,\,\, x \in\Lambda_\varepsilon^c\cap[t_1<|u|<t_2].$$
	Note also that, by the definition of $\overline{F}'_2$,
	$$G''_2(\varepsilon x, u)u^2-\frac{1}{2}G'_2(\varepsilon x,u)u=0,\,\,\,\text{a.e}\,\,\, x \in \Lambda_\varepsilon^c\cap[|u|\geq t_2].$$
Gathering the above information, we derive that $u=0,\,\,\,\text{a.e.}\,\,\,x\in\Lambda_\varepsilon.$ Hence, inasmuch as $u \in \mathcal{N}_\varepsilon$, we get
	$$||u||_{H_\varepsilon}^2+\int_{\mathbb{R}^N}F'_1(u)u=\int_{\Lambda_\varepsilon^c}G'_2(\varepsilon x,u)u\leq b_0\int_{\mathbb{R}^N}|u|^2$$
that leads to $u\equiv 0$, which is absurd because $u \in \mathcal{N}_\varepsilon$, showing the desired result.
	\end{proof}

	In view of the last proposition, we can establish the notion of critical point for $J_\varepsilon|_{\mathcal{N}_\varepsilon}$. Recall that $u \in \mathcal{N}_\varepsilon$ is a critical point of $J_\varepsilon$ constrained to $\mathcal{N}_\varepsilon$ when
	$$||J'_\varepsilon(u)||_*:=\min_{\lambda\in \mathbb{R}}||J'_\varepsilon(u)-\lambda\Psi'_\varepsilon(u)||=0. \quad \mbox{(See \cite[Proposition 5.2]{Willem})}$$
By a $(PS)_c$  sequence associated with $J_\varepsilon|_{\mathcal{N}_\varepsilon}$, we mean a sequence $(u_n)$ in $\mathcal{N}_\varepsilon$ such that
	$$J(u_n)\rightarrow c\,\,\,\,\text{and}\,\,\,\,||J'(u_n)||_*\rightarrow 0.$$
From now on, we say that $J_\varepsilon|_{\mathcal{N}_\varepsilon}$ satisfies the $(PS)$ condition when each $(PS)_c$ sequence for $J_\varepsilon|_{\mathcal{N}_\varepsilon}$ has a convergent subsequence, for any $c \in \mathbb{R}$.
	 
	The next proposition relates critical points of  $J_\varepsilon|_{\mathcal{N}_\varepsilon}$ with  critical points of $J_\varepsilon$ in $X_\varepsilon$.
	
	\begin{proposition}\label{47}
		Let $u \in \mathcal{N}_\varepsilon$ be a critical point of $J_\varepsilon$ constrained to $\mathcal{N}_\varepsilon$. Then $u$ is a critical point of $J_\varepsilon$ on $X_\varepsilon$. 
	\end{proposition} 
	\begin{proof}
		If $u \in \mathcal{N}_\varepsilon$ is a critical point of $J_\varepsilon|_{\mathcal{N}_\varepsilon}$, then
		$$J'_\varepsilon(u)=\lambda\Psi'\varepsilon(u),$$
		for some $\lambda \in \mathbb{R}$. Consequently, 
		$$
		0=J'_\varepsilon(u)u=\lambda\Psi'_\varepsilon(u)u.
		$$
		Since $u \in \mathcal{N}_\varepsilon$, the arguments explored in the proof of Proposition \ref{20} yields $\Psi'_\varepsilon(u)u\neq 0$. Hence, the above equality guarantees that $\lambda=0$ and the proof is over.  
	\end{proof}	

	We finish this subsection by proving that $J_\varepsilon|_{\mathcal{N}_\varepsilon}$ satisfies the $(PS)$ condition.
	
	\begin{proposition}\label{45}
		$J_\varepsilon|_{\mathcal{N}_\varepsilon}$ satisfies the $(PS)$ condition.
	\end{proposition}
	\begin{proof}
		Let $(u_n)$ be an arbitrary $(PS)_c$ sequence for $J_\varepsilon|_{\mathcal{N}_\varepsilon}$. Then,
		$$J_\varepsilon(u_n) \rightarrow c \quad \mbox{and} \quad J'_\varepsilon(u_n)=\lambda_n\Psi'_\varepsilon(u_n)+o_n(1),$$
		for some sequence of real numbers $(\lambda_n)$. Taking into account that $J_\varepsilon(u_n)\rightarrow c$ and $J'_\varepsilon(u_n)u_n=0$, repeating the same reasoning of the proof of Lemma \ref{9}, one has that $(u_n)$ is a bounded sequence. By Corollary \ref{12}, it suffices to show that $(u_n)$ is a $(PS)_c$ sequence for $J_\varepsilon$. Aiming this fact, we will prove that 
		\begin{equation}\label{19}
		\lambda_n \rightarrow 0.
		\end{equation}
		 Note that $(u_n)$ satisfies 
		 $$
		 0=J'_\varepsilon(u_n)u_n=\lambda_n\Psi'_\varepsilon(u_n)u_n+o_n(1).
		 $$
		Arguing as in the proof of Proposition \ref{20}, it is possible to show that if $|\Psi'_\varepsilon(u_n)u_n|=o_n(1)$, then
		 $$\int_{\Lambda_\varepsilon}|u_n|^2\leq o_n(1)\Rightarrow \int_{\Lambda_\varepsilon}|u_n|^2 = o_n(1).$$
		This combined with the boundedness of $(u_n)$ leads to 
		 $$\int_{\Lambda_\varepsilon}|u_n|^p= o_n(1).$$
		 Consequently,
		 $$||u_n||_{H_\varepsilon}^2+\int_{\mathbb{R}^N}F'_1(u_n)u_n=\int_{\Lambda_\varepsilon}F'_2(u_n)u_n+\int_{\Lambda_\varepsilon^{c}}G'_2(\varepsilon x, u_n)\leq o_n(1)+b_0\int_{\mathbb{R}^N}|u_n|^2,$$
		 which combines with (\ref{2}) to give 
		 $$
		 \int_{\mathbb{R}^{N}}(|\nabla u_n|^2+(V(\varepsilon x)+1)|u_n|^2)+\int_{\mathbb{R}^N}F_1(u_n) \leq o_n(1). 
		 $$
		 The above inequality implies that $u_n \to 0$ in $X_\varepsilon$, which contradicts Proposition \ref{20}. Thereby, (\ref{19}) is true and the proof is completed.  	 
	\end{proof}	
	\subsection{Existence of positive solution for $(P_\varepsilon)$}
	 For the goals of this section, we will consider the following autonomous problem
	\begin{equation}
	\left\{\begin{aligned}
	-&\Delta u + V_0u =u\log u^2,\;\;\mbox{in}\;\;\mathbb{R}^{N},\nonumber \\
	&u \in H^{1}(\mathbb{R}^{N})\cap L^{F_1}(\mathbb{R}^N).
	\end{aligned}
	\right.\leqno{(P_0)}
	\end{equation}
	The energy functional related to the $(P_0)$ is given by
	$$
	J_0(u):=\frac{1}{2}\int_{\mathbb{R}^N}(|\nabla u|^2+(V_0+1)|u|^2)+\int_{\mathbb{R}^N}F_1(u)-\int_{\mathbb{R}^N}F_2(u).
	$$
	It is well known (see \cite{Alves-de Morais, Alves-Ji, Squassina-Szulkin}) that $(P_0)$ has a positive ground state solution $u_0$, which satisfies
	$$c_0:=\inf_{u\in \mathcal{N}_0}J_0(u)=J_0(u_0),$$
	where $\mathcal{N}_0$ is the Nehari set associated with $J_0$, i.e.,
	$$\mathcal{N}_0:=\left\{u\in H^{1}(\mathbb{R}^{N})\cap L^{F_1}(\mathbb{R}^N);\,J_0(u)=\frac{1}{2}\int_{\mathbb{R}^N}|u|^2\right\}.
	$$
	Hereafter, we fix
	\begin{equation} \label{X}
	X=\left(H^{1}(\mathbb{R}^{N})\cap L^{F_1}(\mathbb{R}^N),\,\,(||\cdot||_{H^1(\mathbb{R}^N)}+||\cdot||_{L^{F_1}(\mathbb{R}^N)})\right),
	\end{equation}
	where $||\cdot||_{H^1(\mathbb{R}^N)}$ denotes the usual norm in $H^1(\mathbb{R}^N)$.
	
	The level $c_0$ can be characterized by 
	$$c_0=\inf_{u\in \mathcal{N}_0}J_0(u)=\inf_{u \in (X-\{0\})}\max_{t \geq 0}J_0(tu).$$
	
	In the next lemma we prove that the solution $u_\varepsilon$ obtained in Theorem \ref{53} is a ground state solution of $(\tilde{S}_\varepsilon)$, and we study the behavior of levels $c_\varepsilon$, as $\varepsilon \rightarrow 0^+$. By a ground state solution we mean a solution of least energy of $(\tilde{S}_\varepsilon)$, that is, a solution verifying
	$$\displaystyle{\inf_{u\in\mathcal{N}_\varepsilon}}J_\varepsilon(u)=J_\varepsilon(u_\varepsilon).$$
	
	\begin{lemma}\label{29}
		The following properties hold:
		\begin{itemize}
			\item [$i)$] There is $\gamma_0>0$ such that $c_\varepsilon\geq \gamma_0$ for all $\varepsilon>0$.
			\item [$ii)$] $c_\varepsilon=\displaystyle{\inf_{u\in\mathcal{N}_\varepsilon}}J_\varepsilon(u)$ for all $\varepsilon >0$.
			\item [$iii)$] $\displaystyle{\limsup_{\varepsilon\rightarrow 0}}\,\,c_\varepsilon \leq c_0$.
		\end{itemize}
	\end{lemma}
	\begin{proof}
		$i)$ Note that
		$$J_\varepsilon(u)\geq \frac{1}{2}\int_{\mathbb{R}^N}(|\nabla u|^2+(\alpha_0+1)|u|^2)+\int_{\mathbb{R}^N}F_1(u)-\int_{\mathbb{R}^N}F_2(u),$$
		with $\alpha_0=\displaystyle{\inf_{\mathbb{R}^N}}\,V$. Arguing as Lemma \ref{16}-$i)$, we find $r_0\approx 0^+$ and $\gamma_0>0$ independent of $\varepsilon$ such that
		$$J_\varepsilon(u)\geq \rho_0,\,\,\,\, \forall u \in X_\varepsilon, ||u||_\varepsilon=r_0.$$
		By the definition of $c_\varepsilon$, we derive $c_\varepsilon\geq \gamma_0$.\\
		$ii)$ By Lemma  \ref{27} we know that $u \in O_\varepsilon$ for each $u \in \mathcal{N}_\varepsilon$. In this way, using the same ideas of Theorem \ref{16}-$ii)$, there is $t_0$ such that $J_\varepsilon(t_0u)<0$. Setting $\eta:[0,1]\longrightarrow X_\varepsilon$ given by $\eta(t):=t(t_0u)$, it follows that $\eta \in \Gamma_\varepsilon$, and so, 
		$$c_\varepsilon\leq \max_{t \in [0,1]} J_\varepsilon(\eta(t))\leq \max_{s \geq 0} J_\varepsilon(su) \leq J_\varepsilon(u).$$
		The above inequality shows that
		$$c_\varepsilon \leq \displaystyle{\inf_{u\in\mathcal{N}_\varepsilon}}J_\varepsilon(u).$$
		The reverse inequality follows by observing that
		$$\displaystyle{\inf_{u\in\mathcal{N}_\varepsilon}}J_\varepsilon(u)\leq J_\varepsilon(u_\varepsilon)=c_\varepsilon.$$
		$iii)$: Let $u_0 \in \mathcal{N}_0$ be a positive ground state solution of $(P_0)$, i.e, 
		$$
		J_0(u_0)=c_0 \quad \mbox{and} \quad J'_0(u_0)=0.
		$$
		
		For each $R>0$, set $\phi_R(x):=\phi(\frac{1}{R}x)$, where $\phi \in C_0^\infty(\mathbb{R}^N)$ is such that $\phi(x)=1$, for $x \in B_1(0)$, and $\phi(x) = 0 $, for $x \in B_2^c(0)$. Then, putting $u_R:=\phi_Ru_0$, it is easy to check that
		$$
		u_R \rightarrow u_0\,\,\, \text{in}\,\,\, H^1(\mathbb{R}^N) \quad \mbox{as} \quad R\rightarrow \infty.
		$$
	Since $0\leq u_R\leq u_0$, the Lebesgue Dominated Convergence Theorem ensures that
		$$
		\int_{\mathbb{R}^N}F_1(u_R)\longrightarrow \int_{\mathbb{R}^N}F_1(u_0),\,\,\,\text{as}\,\,\,R\rightarrow \infty.
		$$
By the last two limits we can infer that $u_R\rightarrow u_0$ in $X$.
		
		Given $R>0$, from the definition of $u_R$, one can see that $u_R \in O_\varepsilon$ for each $\varepsilon>0$, since $u_0>0$ and $0\in \Lambda_\varepsilon$. So, thanks to preceding item, we find $t_\varepsilon>0$ in such way that 
		$$c_\varepsilon\leq \max_{t \geq 0}J_\varepsilon(tu_R)=J_\varepsilon(t_\varepsilon u_R).$$
		Our next step is to show that, for some $\varepsilon_0>0$, the family $(t_\varepsilon)_{0<\varepsilon<\varepsilon_0}$ is bounded. In fact, as $t_\varepsilon u_R \in \mathcal{N}_\varepsilon$, 
		$$\int_{\mathbb{R}^N}(|\nabla u_R|^2+(V(\varepsilon x)+1)|u_R|^2)=\frac{1}{t_\varepsilon}\int_{\Lambda_\varepsilon}F'_2(t_\varepsilon u_R)u_R+\frac{1}{t_\varepsilon}\int_{\Lambda_\varepsilon^c}\tilde{F}'_2(t_\varepsilon u_R)u_R-\frac{1}{t_\varepsilon}\int_{\mathbb{R}^N}F'_1(t_\varepsilon u_R)u_R.$$
		Considering that $u_R\equiv 0$ in $B^c_{2R}(0)$ and $V(\varepsilon x) \rightarrow V(0)=V_0$, we have
		$$
		\int_{\mathbb{R}^N}(|\nabla u_R|^2+(V(\varepsilon x)+1)|u_R|^2)\longrightarrow \int_{\mathbb{R}^N}(|\nabla u_R|^2+(V_0+1)|u_R|^2),
		$$
		as $\varepsilon \rightarrow 0$, for each $R>0$. On the other hand, if $t_\varepsilon \rightarrow \infty$ as $\varepsilon \rightarrow 0$, the following claim holds:
		\begin{claim}\label{57}
		$$
		\left(\frac{1}{t_\varepsilon}\int_{\Lambda_\varepsilon}F'_2(t_\varepsilon u_R)u_R+\frac{1}{t_\varepsilon}\int_{\Lambda_\varepsilon^c}\tilde{F}'_2(t_\varepsilon u_R)u_R-\frac{1}{t_\varepsilon}\int_{\mathbb{R}^N}F'_1(t_\varepsilon u_R)u_R\right)\,\,\,\longrightarrow\,\,\,\infty.
		$$
		\end{claim}
		First of all, the limit  $\chi_{\Lambda_\varepsilon}(x)\rightarrow 1$ as $\varepsilon \rightarrow 0^+$ together with $(A_1)$ guarantees that
		$$\frac{1}{t_\varepsilon}\int_{\Lambda_\varepsilon^c}\tilde{F}'_2(t_\varepsilon u_R)u_R=o_\varepsilon(1).$$
		Thereby, in order to get the Claim \ref{57}, it suffices to show that
		$$A_\varepsilon:=\left(\frac{1}{t_\varepsilon}\int_{\Lambda_\varepsilon}F'_2(t_\varepsilon u_R)u_R-\frac{1}{t_\varepsilon}\int_{\mathbb{R}^N}F'_1(t_\varepsilon u_R)u_R\right)\,\,\,\longrightarrow\,\,\,\infty.$$
		Observe that, by (\ref{3}),
		$$\begin{aligned} A_\varepsilon=&\int_{\mathbb{R}^N}|u_R|^2+\int_{\mathbb{R}^N}|u_R|^2\log (t_\varepsilon|u_R|)^2-\frac{1}{t_\varepsilon}\int_{\Lambda_\varepsilon^c}F'_2(t_\varepsilon u_R)u_R=\\
		=&\log (t_\varepsilon)^2\int_{\mathbb{R}^N}|u_R|^2-\frac{1}{t_\varepsilon}\int_{\Lambda_\varepsilon^c}F'_2(t_\varepsilon u_R)u_R+C_R,
		\end{aligned}
		$$
		with $C_R=\displaystyle{\int_{\mathbb{R}^N}}(|u_R|^2+|u_R|^2\log |u_R|^2)$. From the definition of $F_2$, 
		$$
		\frac{1}{t_\varepsilon}F'_2(t_\varepsilon u_R)u_R=u_R^2\log (t_\varepsilon|u_R|)^2-\log\delta^2u_R^2+\frac{2\delta}{t_\varepsilon}u_R-2u_R^2,
		$$
		and so,
		$$
		\frac{1}{t_\varepsilon}\int_{\Lambda_\varepsilon^c}F'_2(t_\varepsilon u_R)u_R\leq\int_{\Lambda_\varepsilon^c}u_R^2\log (t_\varepsilon|u_R|)^2+\frac{2\delta}{t_\varepsilon}\int_{\mathbb{R}^N}u_R+B_R,
		$$
		with $B_R:=-\log\delta^2\displaystyle{\int_{\mathbb{R}^N}}u_R^2$. From this and using that $t_\varepsilon \rightarrow \infty$ as $\varepsilon \rightarrow 0$, one finds 
		$$
		A_\varepsilon\geq\log (t_\varepsilon)^2\int_{\mathbb{R}^N}|u_R|^2-\int_{\Lambda_\varepsilon^c}u_R^2\log (t_\varepsilon|u_R|)^2+o_
		\varepsilon(1)+D_R,
		$$
		where $D_R=C_R-B_R$. Therefore,
		$$A_\varepsilon\geq\log (t_\varepsilon)^2\int_{\Lambda_\varepsilon}|u_R|^2-\int_{\Lambda_\varepsilon^c}u_R^2\log |u_R|^2+o_
		\varepsilon(1)+D_R,
		$$
	from where it follows that
		$$A_\varepsilon\rightarrow \infty\,\,\,\text{as}\,\,\,\varepsilon \rightarrow 0^+,$$
		showing the Claim \ref{57}.
		
		As a byproduct of the Claim \ref{57}, we get that $(t_\varepsilon)_{0<\varepsilon<\varepsilon_0}$ is bounded, for some $\varepsilon_0>0$. Now, take $t_R>0$ such that $J_0(t_Ru_R)=\max_{t \geq 0}J_0(tu_R)$. Note that
		$$J_\varepsilon(t_\varepsilon u_R)-J_0(t_\varepsilon u_R)=\frac{t_\varepsilon^2}{2}\int_{\mathbb{R}^N}(V(\varepsilon x)-V_0)|u_R|^2+\int_{\Lambda_\varepsilon^c}(F_2(t_\varepsilon u_R)-\tilde{F}_2(t_\varepsilon u_R)).$$
		Using that $u_R$ has compact support, $u_R\rightarrow u_0$ in $X$ as $R\rightarrow \infty$ and the Lebesgue's Dominated Convergence Theorem, we arrive at
		$$J_\varepsilon(t_\varepsilon u_R)-J_0(t_\varepsilon u_R)=o_\varepsilon(1),$$
		\begin{equation}\label{28}
		\limsup_{\varepsilon\rightarrow 0}\,c_\varepsilon \leq \limsup_{\varepsilon\rightarrow 0}\,J_\varepsilon(t_\varepsilon u_R)\leq J_0(t_R u_R).
		\end{equation}
		The choose of $t_R$ gives $t_R \rightarrow 1$ (see \cite[Lemma 3.7]{Alves-Ji}), and then, 
		$$J_0(t_R u_R)\rightarrow J_0(u_0)=c_0,\,\,\,\text{as}\,\,\,R\rightarrow \infty.$$
		The result is a direct consequence of the  limit above and (\ref{28}).	
	\end{proof}	
	
Now, we are ready to prove the existence of positive ground state solution for $(\tilde{S}_\varepsilon)$.
	
	\begin{proposition}\label{54}
		Given $\varepsilon>0$ the problem $(\tilde{S}_\varepsilon)$ has a positive ground state solution.
	\end{proposition}
	\begin{proof}
		Let $u_\varepsilon$ be the solution of $(\tilde{S}_\varepsilon)$ given in Theorem \ref{53}. For $v\in X_\varepsilon$, set $v^+:=\max\{v,0\}$ and $v^-:=\max\{0,-v\}$. Therefore, either $u_\varepsilon^+=0$ or $u_\varepsilon^-=0$, otherwise we would have
		$u_\varepsilon^+$, $u_\varepsilon^- \in \mathcal{N}_\varepsilon$ and $J_\varepsilon(u_\varepsilon)=J_\varepsilon(u_\varepsilon^+)+J_\varepsilon(u_\varepsilon^-)\geq 2c_\varepsilon$, which contradicts $J_\varepsilon(u_\varepsilon)=c_\varepsilon$. Thereby, since $g$ is odd, we may assume that $u_\varepsilon$ is a nonnegative solution of $(\tilde{S}_\varepsilon)$. 		
		By an analogous reasoning as used in the proof of \cite[Theorem 3.1]{Alves-Ji} and \cite[Section 3.1]{d'Avenia}, using a suitable version of maximum principle (\cite[Theorem 1]{Vazquez}), we deduce that $u_\varepsilon$ is positive in whole $\R^N$. 
	\end{proof}
	
	Our next result improves \cite[Lemma 3.9]{Alves-Ji} and it is an essential step in order to get a solution for $(S_\varepsilon)$. 
	
	\begin{lemma}\label{38}
		Let $(u_n)$ be a nonnegative sequence with $u_n\in X_{\varepsilon_n}$, $J_{\varepsilon_n}(u_n)= c_{\varepsilon_n}$, $J'_{\varepsilon_n}(u_n)=0$ and $\varepsilon_n \rightarrow 0$. Then, there exits a sequence $(y_n) \subset \mathbb{R}^N$ such that $w_n(x):=u_n(x+y_n)$ has a convergent subsequence, $\sup_{n\in \mathbb{N}}||w_n||_\infty<\infty$
		and
	\begin{equation} \label{39}
			w_n(x) \rightarrow 0\,\,\,\,\text{as}\,\,\,\,|x|\rightarrow \infty \quad  \mbox{uniformly in} \quad n\in\mathbb{N}.
	\end{equation}
		Furthermore, for some $y_0 \in \Lambda$, the following limit holds $\displaystyle \lim_{n \to +\infty}(\varepsilon_n y_n) = y_0$. 
		\end{lemma}

	\begin{proof}
		To begin with, note that $(u_n)$ is a bounded sequence in the space $X$ given in (\ref{X}). Indeed, by the assumptions and employing Lemma \ref{29}-$iii)$, $(u_n)$ must satisfy
		$$
		J_{\varepsilon_n}(u_n) \leq M_1\,\,\,\,\text{and}\,\,\,\,J'_{\varepsilon_n}(u_n)u_n = 0, \quad \forall n \in \mathbb{N},
		$$
		for some positive $M_1$. By following closely the arguments of Lemma \ref{9}, we find, instead of (\ref{7}),
		$$M_1\geq\frac{1}{2}\int_{\Lambda_{\varepsilon_n}}|u_n|^2.$$
		Hence, by the same ideas explored in the proof of Lemma \ref{9}, there are a $M_1,M_2>0$ such that
		$$\int_{\Lambda_{\varepsilon_n}}|u_n|^2\log |u_n|^2\leq M_2(1+||v_n||_{H_{\varepsilon_n}}^{1+r})$$
		and, 
		$$
		M_1+M_2(1+||v_n||_{H_{\varepsilon_n}}^{1+r})\geq C||u_n||_{H_{\varepsilon_n}}^2+\int_{\Lambda_\varepsilon^c}F_1(u_n)\geq C||u_n||_{H_{\varepsilon_n}}^2, \quad \forall n \in \mathbb{N},
		$$
		for some $C>0$ and $0<r<1$, which shows the boundedness of $(||u_n||_{H_{\varepsilon_n}})$ in $\mathbb{R}$. Now, the conditions on $V$ ensure that $(u_n)$ is bounded in $H^1(\mathbb{R}^N)$. Since
		$$\int_{\mathbb{R}^N}F_1(u_n)=J_{\varepsilon_n}(u_n)-\frac{1}{2}||u_n||_{H_{\varepsilon_n}}^2+\int_{\mathbb{R}^N}G_2(\varepsilon_n x, u_n),$$
		we infer that
		$$\sup_{n\in \mathbb{N}} \int_{\mathbb{R}^N}F_1(u_n)<\infty,$$
	proving the boundedness of $(u_n)$ in $X$. 
		For some $r, \lambda>0$ and a sequence $(y_n)$ it holds
		\begin{equation}\label{30}
		\limsup_{n \to +\infty} \int_{B_r(y_n)}|u_n|^2 \geq \lambda >0.
		\end{equation}
		Otherwise, using a concentration-compactness principle due to Lions (\cite[Lemma 1.21]{Willem}), we would have
		$$u_n \rightarrow 0\,\,\,\text{in}\,\,\, L^p(\mathbb{R}^N) \quad \forall p\in (2,2^*),$$
		then 
		$$\int_{\mathbb{R}^N}G'_2(\varepsilon_n x, u_n)u_n = o_n(1)\,\,\,\text{and}\,\,\,\int_{\mathbb{R}^N}G_2(\varepsilon_n x, u_n)=o_n(1).$$
		From the assumptions in the statement we get $J'_{\varepsilon_n}(u_n)u_n = 0$. This associated with the last equality give 
		$$ o_n(1) = ||u_n||_{H_{\varepsilon_n}}^2+\int_{\mathbb{R}^N}F'_1(u_n)u_n.$$
		The above limit together with (\ref{2}) ensures that 
		$$
		||u_n||_{H_{\varepsilon_n}}^2+\int_{\mathbb{R}^N}F_1(u_n) \to 0,
		$$
		which permits to conclude that $J_{\varepsilon_n}(u_n)=c_{\varepsilon_n} \rightarrow 0$, contradicting Lemma \ref{29}-$i)$.
		
		From now on, set $w_n:=u_n(\cdot+y_n)$. The boundedness of $(u_n)$ and (\ref{30}) yield that $(w_n)$ is a bounded sequence in $X$, and so, we may assume that there is $w \in X -\{0\}$ such that 
		$$w_n\rightharpoonup w\,\,\,\,\text{in}\,\,\,\, X.$$
	 Our next step is proving that $(\varepsilon_n y_n)$ is a bounded sequence in $\mathbb{R}^N$. This fact is a direct consequence of the claim below.
		
		\begin{claim}\label{32}
			 It holds $\displaystyle \lim_{n \to +\infty} d(\varepsilon_n y_n,\,\overline{\Lambda})= 0$, with $d$ being the usual distance between $\varepsilon_n y_n$ and $\overline{\Lambda}$ in $\R^N$.
		\end{claim}
		
		The proof of the claim follows the same ideas of \cite[Claim 3.1]{Alves-Ji}, however for the reader's convenience we will write its proof. Arguing by contradiction, if the claim is not true, there exist some subsequence of $(\varepsilon_n y_n)$, still denoted by itself, and $\gamma >0$ satisfying
		\begin{equation*}
		d(\varepsilon_n y_n,\,\overline{\Lambda})\geq \gamma,\,\,\,\forall n \in \mathbb{N}. 
		\end{equation*}
		Then, for some $r>0$, 
	$$
	B_r(\varepsilon_n y_n) \subset \Lambda^c,\,\,\,\forall n \in \mathbb{N}.
	$$

		Now, we fix  ,for each $j \in \mathbb{N}$, $v_j = \phi_jw$, with $\phi_j$ defined as in Lemma \ref{29}-$iii)$. So, we know that $v_j \rightarrow w$ in $X$. For each $j$ fixed, a simple change of variable leads to 
		\begin{equation}\label{31}
		\int_{\mathbb{R}^N}(\nabla w_n \nabla v_j + (V(\varepsilon_n x+\varepsilon_n y_n))w_nv_j)+\int_{\mathbb{R}^N}F'_1(w_n)v_j=\int_{\mathbb{R}^N}G'_2(\varepsilon_n x, w_n)v_j.
		\end{equation}
		Writing
		$$\int_{\mathbb{R}^N}G'_2(\varepsilon_n x, w_n)v_j=\int_{B_{\frac{r}{\varepsilon_n}}(0)}G'_2(\varepsilon_n x, w_n)v_j + \int_{B^c_{\frac{r}{\varepsilon_n}}(0)}G'_2(\varepsilon_n x, w_n)v_j$$
		and using $(A_1)$, we find
		$$\int_{\mathbb{R}^N}G'_2(\varepsilon_n x, w_n)v_j\leq b_0\int_{B_{\frac{r}{\varepsilon_n}}(0)}w_nv_j + \int_{B^c_{\frac{r}{\varepsilon_n}}(0)}F'_2(w_n)v_j,$$
		and so
		$$\int_{\mathbb{R}^N}(\nabla w_n \nabla v_j + Cw_nv_j)+\int_{\mathbb{R}^N}F'_1(w_n)v_j \leq \int_{B^c_{\frac{r}{\varepsilon_n}}(0)}F'_2(w_n)v_j,$$
		for a  convenient $C>0$. Since $v_j$ has compact support, one can sees that
		$$
		\int_{B^c_{\frac{r}{\varepsilon_n}}(0)}F'_2(w_n)v_j \longrightarrow 0 \quad \mbox{as} \quad n \rightarrow \infty.
		$$ 
		By using that $w_n \rightharpoonup w$ in $X$, we firstly take the limit of $n \rightarrow \infty$ and after the limit of $j \rightarrow \infty$ to get the inequality below 
		$$\int_{\mathbb{R}^N}(|\nabla w|^2+ C|w|^2)+\int_{\mathbb{R}^N}F'_1(w)w\leq 0,$$
		which yields $w=0$. This contradiction proves the claim.
		
		The preceding claim ensures that, going to a subsequence if necessary, $\varepsilon_n y_n \rightarrow y_0 \in \overline{\Lambda}$ for some $y_0$. Actually, we will prove that $y_0 \in \Lambda$. To this aim, note that for each $R>0$ the sequence $\chi_n(x):=\chi_\Lambda(\varepsilon_n x+\varepsilon_n y_n)$ is a bounded sequence in $L^q(B_R(0))$, for any $q \in [2,\infty)$. Since $L^q(B_R(0))$ is a reflexive space for all $q \in [2,\infty)$, then there exists a function $\chi_R \in L^q(B_R(0))$ such that
		$$\chi_n\rightharpoonup\chi_R\,\,\,\text{in}\,\,\,L^q(B_R(0)).$$
		The reader is invited to note that, given positive numbers $0<R_1<R_2$, the functions $\chi_{R_1}$ and $\chi_{R_2}$ obtained in the same way of $\chi_R$ satisfy 
		$$\chi_{R_1}\equiv \chi_{R_2}|_{B_{R_1}(0)}.$$
		Therefore, there is a measurable function $\chi \in L^q_{loc}(\mathbb{R}^N)$ satisfying
		\begin{equation}\label{33}
		\chi_n\rightharpoonup\chi \,\,\,\text{in}\,\,\,L^q(B_R(0)),
		\end{equation}
		for each $R>0$. Note also that $0\leq \chi \leq 1$. 
		
		In the same way of (\ref{31}), for each $\phi \in C_0^\infty(\mathbb{R}^N)$ we have
		$$\int_{\mathbb{R}^N}(\nabla w_n \nabla \phi + (V(\varepsilon_n x+\varepsilon_n y_n)+1)w_n\phi)+\int_{\mathbb{R}^N}F'_1(w_n)\phi=\int_{\mathbb{R}^N}G'_2(\varepsilon_n x+\varepsilon_n y_n, w_n)\phi.$$
		By Claim \ref{32} and  (\ref{33}), 
		$$\int_{\mathbb{R}^N}(\nabla w \nabla \phi + (V(y_0)+1)w\phi)+\int_{\mathbb{R}^N}F'_1(w)\phi=\int_{\mathbb{R}^N}\tilde{G}'_2(x, w)\phi,$$
		where
		$$ \tilde{G}'_2(z, t):= \chi(z)F'_2(t)+(1-\chi(z))\tilde{F}'_2(t).$$
	It is easy to check that $\tilde{G}'_2$ satisfies
		$$\tilde{G}'_2(z, t)\leq C(|t|+|t|^{p-1}),$$
		where $p \in (2,2^*)$. Moreover, the map $ t\longmapsto \frac{\tilde{G}'_2(z, t)}{t}$, for $t> 0$, is an nondecreasing function.
		
		The above arguments guarantee that $\tilde{J}'(w)=0$, where $\tilde{J}:X\longrightarrow \mathbb{R}$ is the functional given by
		$$
		\tilde{J}(u):=\frac{1}{2}\int_{\mathbb{R}^N}(|\nabla u|^2+(V(y_0)+1)|u|^2)+\int_{\mathbb{R}^N}F_1(u)-\int_{\mathbb{R}^N}\tilde{G}_2(x,u),
		$$
		and $\tilde{G}_2(x,u):=\displaystyle{\int_{0}^{t}}\tilde{G}'_2(x,s)\,ds$. Next, we set 
		$$J_{V(y_0)}:=\frac{1}{2}\int_{\mathbb{R}^N}(|\nabla u|^2+(V(y_0)+1)|u|^2)+\int_{\mathbb{R}^N}F_1(u)-\int_{\mathbb{R}^N}F_2(u) \quad \forall u \in X, $$
		
		$$ \mathcal{M}_0:=\left\{u \in X-\{0\};\,J'_{V(y_0)}(u)u=0\right\}$$
		and
		$$c_{V(y_0)}=\inf_{u \in \mathcal{M}_0} J_{V(y_0)}(u)=\inf_{u\in X-\{0\}}\left\{\max_{t\geq 0}J(tw)\right\}.$$
		Define also $\Sigma_0:=\text{supp}\chi$ and $\mathcal{O}_0:=\{u \in X_\varepsilon;\, |\text{supp}(|u|)\cap\Sigma_0|>0\}$. Using the same ideas explored in the proof of Lemma \ref{16}, the conditions on $\tilde{G}_2$ allow us to conclude that 
		$$
		\tilde{J}(tv)\rightarrow -\infty,\,\,\,\text{as}\,\,\,t\rightarrow \infty,
		$$
		for each $v \in \mathcal{O}_0$. Since $w\neq 0$ and $\tilde{J}'(w)=0$, we get $w \in \mathcal{O}_0$. Therefore, by standard arguments,
		$$\tilde{J}(w)=\max_{t\geq0}\tilde{J}(tw)\geq\max_{t\geq0}J(tw)\geq c_{V(y_0)}.$$
		In the same way of (\ref{6}), we find by a change of variable,
		$$
		\begin{aligned}
		c_{\varepsilon_n}&={J}_{\varepsilon_n}(u_n)-\frac{1}{2}{J}'_{\varepsilon_n}(u_n)u_n=\\
		&=\frac{1}{2}\int_{\mathbb{R}^N}(|w_n|^2+[F_2(w_n)-\frac{1}{2}F'_2(w_n)w_n+\frac{1}{2}G'_2(\varepsilon_n x+\varepsilon_n y_n,w_n)w_n-G_2(\varepsilon_n x+\varepsilon_n y_n,w_n)]).
		\end{aligned}
	$$
From $(A_1)-iv)$, 
 
	$$
	c_{\varepsilon_n} \geq 
	\frac{1}{2}\int_{B_R(0)}(|w_n|^2+[F_2(w_n)-\frac{1}{2}F'_2(w_n)w_n+\frac{1}{2}G'_2(\varepsilon_n x+\varepsilon_n y_n,w_n)w_n-G_2(\varepsilon_n x+\varepsilon_n y_n,w_n)])
	$$
for each $R>0$. Now, fix $p \in (2,2^*)$. Since $w_n \rightarrow w$ in $L^p(B_R(0))$, the growth conditions on $F'_2$ and $\tilde{F}'_2$ assures that, for some $q \in (p,2^*)$, it holds
		$$\left\{\begin{aligned}
		&F'_2(w_n)w_n \rightarrow F'_2(w)w,\,\,\,\text{in}\,\,\,L^{\frac{q}{p}}(B_R(0));\\
		&\tilde{F}'_2(w_n)w_n \rightarrow \tilde{F}'_2(w)w,\,\,\,\text{in}\,\,\,L^{\frac{q}{p}}(B_R(0)).
		\end{aligned}
		\right.
		$$
		The convergence in (\ref{33}) implies that $\chi_n \rightharpoonup \chi$ in $L^r(B_R(0))$, where $r$ is the conjugate exponent of $q/p$. Gathering these information, 
		$$\chi_nF'_2(w_n)+(1-\chi_n)\tilde{F}'_2(w_n) \longrightarrow \chi F'_2(w)+(1-\chi)\tilde{F}'_2(w) \,\,\,\text{in}\,\,\,L^1(B_R(0)).$$
Now, employing the fact that 
	$$
G'_2(\varepsilon_n x+\varepsilon_n y_n, w_n)=\chi_n(x) F'_2(w_n)+(1-\chi_n(x))\tilde{F}'_2(w_n),
	$$
we conclude that 
		$$G'_2(\varepsilon_n x+\varepsilon_n y_n, w_n)\rightarrow \tilde{G}'_2(x,w) \,\,\,\text{in}\,\,\,L^1(B_R(0)).$$
		Using an analogous reasoning we also derive
		$$G_2(\varepsilon_n x+\varepsilon_n y_n, w_n)\rightarrow \tilde{G}_2(x,w) \,\,\,\text{in}\,\,\,L^1(B_R(0)).$$
		Consequently,  by Fatou's Lemma (recall the inequality in $(A_1)-iv)$) and Lemma \ref{29},
		$$\begin{aligned}
		c_0&\geq\int_{B_R(0)}\left(\frac{1}{2}|w|^2+[F_2(w)-\frac{1}{2}F'_2(w)w+\frac{1}{2}\tilde{G}'_2(x,w)w-\tilde{G}_2(x,w)]\right), \quad \forall R>0.
		\end{aligned}$$
		Letting $R\rightarrow \infty$, one gets
		$$\begin{aligned}
		c_0&\geq\int_{\mathbb{R}^N}\left(\frac{1}{2}|w|^2+[F_2(w)-\frac{1}{2}F'_2(w)w+\frac{1}{2}\tilde{G}'_2(x,w)w-\tilde{G}_2(x,w)]\right)=\\
		&=\tilde{J}(w)-\frac{1}{2}\tilde{J}'(w)w=\tilde{J}(w)\geq c_{V(y_0)}.
		\end{aligned}
		$$
		By the definitions of levels $c_0$ and $c_{V(y_0)}$, the above inequality ensures that $V(y_0)\leq V(0)=\displaystyle{\inf_{x \in \Lambda}} V(x)$. Thus, by $(V_2)$, we must have $V(y_0)=V(0)=V_0$ and
		$y_0 \in \Lambda.$
		
		In order to finish the proof, it remains to prove that 
		$$w_n \longrightarrow w\,\,\,\,\text{in}\,\,\,\,X \quad \mbox{as} \quad n \to +\infty.$$
		Aiming this goal, we will prove the following result
		\begin{claim}\label{36}
			$\displaystyle \lim_{n \to +\infty} \displaystyle{\int_{(\Lambda_{\varepsilon_n}-y_n)}}|w_n|^2=\int_{\mathbb{R}^N}|w|^2.$
		\end{claim}

			Note first that, since $\varepsilon_n y_n\rightarrow y_0 \in \Lambda$, there exists a number $r>0$ such that
			$$ B_r(\varepsilon_n y_n) \subset \Lambda,$$
			for all $n$ large enough. Thereby, 
			$$B_\frac{r}{\varepsilon_n}(0) \subset \Lambda_{\varepsilon_n}-y_n,$$
			for all $n$ large enough, and so, 
			\begin{equation}\label{34}
			\chi_{(\Lambda_{\varepsilon_n}-y_n)}(x)\longrightarrow 1,\,\,\text{a.e.}\,\,\,x\in \mathbb{R}^N. 
			\end{equation}
			Now, note that, by using $\tilde{G}'_2 \leq F'_2$ and that $\tilde{J}'(w)w=0$ we get $J'_{V(y_0)}(w)w\leq 0$, so that $J'_0(w)w\leq 0$, because $V(y_0)=V_0$. Therefore, for some $t_0 \in (0,1]$ it holds $t_0 w\in \mathcal{N}_0$. Then, from (\ref{34}) and Lemma \ref{29}-$iii)$, 
			\begin{equation}\label{35}
			\begin{aligned}
			c_0\leq J_0(t_0w)=\frac{t_0^2}{2}\int_{\mathbb{R}^N}|w|^2\leq\frac{t_0^2}{2}\liminf_{n \to +\infty}\int_{(\Lambda_{\varepsilon_n}-y_n)}|w_n|^2&\leq \frac{t_0^2}{2}\limsup_{n \to +\infty}\int_{(\Lambda_{\varepsilon_n}-y_n)}|w_n|^2\leq\\
			&\leq\frac{t_0^2}{2}\limsup_{n \to +\infty} c_{\varepsilon_n}\leq c_0,
			\end{aligned}
			\end{equation}
			where we have used that 
			$$
			\frac{1}{2}\int_{(\Lambda_{\varepsilon_n}-y_n)}|w_n|^2=\frac{1}{2}\int_{\Lambda_{\varepsilon_n}}|u_n|^2\leq {J}_{\varepsilon_n}(u_n)-\frac{1}{2}{J}'_{\varepsilon_n}(u_n)u_n=c_{\varepsilon_n}.
			$$
			The above computations prove the claim.
			 
			Observe that the sentence in (\ref{35}) also ensures that $t_0=1$, and so, $w \in \mathcal{N}_0$. Using that $J'_{\varepsilon_n}(u_n)u_n=0$, by a change of variable, we find
			\begin{equation}\label{37}
			\begin{aligned}
			\int_{\mathbb{R}^N}(&|\nabla w_n |^2+ (V(\varepsilon_n x+\varepsilon_n y_n)+1)|w_n|^2)+\int_{\mathbb{R}^N}F'_1(w_n)w_n=\\
			&\int_{(\Lambda_{\varepsilon_n}-y_n)}F'_2(w_n)w_n+\int_{(\Lambda_{\varepsilon_n}-y_n)^c}\tilde{F}'_2(w_n)w_n.
			\end{aligned}
			\end{equation}
			By applying Claim \ref{36} and interpolation, 
			$$\chi_{(\Lambda_{\varepsilon_n}-y_n)}w_n \longrightarrow w\,\,\,\,\text{in}\,\,\,\,L^p(\mathbb{R}^N)$$
		and
			$$ \int_{(\Lambda_{\varepsilon_n}-y_n)}F'_2(w_n)w_n=\int_{\mathbb{R}^N}F'_2(w)w+o_n(1).$$
		As $w \in \mathcal{N}_0$ and 
			$$(V(\varepsilon_n x+\varepsilon_n y_n)+1])|w_n|^2-\tilde{F}'_2(w_n)w_n)\geq 0 \quad \mbox{in} \quad (\Lambda_{\varepsilon_n}-y_n)^c,$$
		the equality (\ref{37}) yields that  
			$$\begin{aligned}
			&\int_{\mathbb{R}^N}(|\nabla w |^2+ (V(y_0)+1)|w|^2)+\int_{\mathbb{R}^N}F'_1(w)w\leq\\
			&\leq\liminf\int_{\mathbb{R}^N}\left(|\nabla w_n |^2+ \int_{(\Lambda_{\varepsilon_n}-y_n)}(V(\varepsilon_n x+\varepsilon_n y_n)+1)|w_n|^2)+\int_{\mathbb{R}^N}F'_1(w_n)w_n\right)\leq\\
			&\leq\int_{\mathbb{R}^N}(|\nabla w |^2+ (V_0+1)|w_n|^2)+\int_{\mathbb{R}^N}F'_1(w)w.
			\end{aligned}
			$$
			Taking into account $V(y_0)=V_0$, we derive that
			$$||w_n||_{H^1(\mathbb{R}^N)}^2\rightarrow ||w||_{H^1(\mathbb{R}^N)}^2\,\,\,\text{and}\,\,\,\int_{\mathbb{R}^N}F'_1(w_n)w_n \rightarrow \int_{\mathbb{R}^N}F'_1(w)w.$$
		The above limit together with (\ref{2}) ensure that $w_n \rightarrow w$ in $X$. Finally the boundedness of $(w_n)$ in $L^{\infty}(\Omega)$ and the limit (\ref{39}) follow as in \cite[Lemma 3.10]{Alves-Ji}
		\end{proof}	

		As a direct consequence of the computations made above, see the sentence (\ref{35}), we have the following result
		
		\begin{corollary}\label{44}
			The levels $c_\varepsilon$ satisfies $\displaystyle \lim_{\varepsilon\rightarrow 0}\,c_\varepsilon=c_0.$
		\end{corollary}

		Finally, we are ready to prove that $(P_\varepsilon)$ has a positive solution for all $\varepsilon$ small enough.
		
		\begin{theorem}\label{50}
			There exists $\varepsilon_0>0$ such that $(S_\varepsilon)$ (and so $(P_\varepsilon)$) has a positive solution $u_\varepsilon \in X_\varepsilon$ for all $\varepsilon \in (0,\varepsilon_0)$.
		\end{theorem}
		\begin{proof}
		In what follows, we will prove that
			\begin{equation}\label{40}
			u_\varepsilon(x)< t_1,\,\,\,\forall x \in \mathbb{R}^N-\Lambda_\varepsilon,
			\end{equation}
			for $\varepsilon \in (0,\varepsilon_0)$. Indeed, otherwise there must be sequences $\varepsilon_n \rightarrow 0$ and $(u_{\varepsilon_n})$ such that $J_{\varepsilon_n}(u_{\varepsilon_n})=c_{\varepsilon_n}$ and $J'_{\varepsilon_n}(u_{\varepsilon_n})=0$, but $u_{\varepsilon_n}$ does not satisfy (\ref{40}). By Lemma \ref{38}, there exists a sequence $(y_n)$ in $\mathbb{R}^N$ satisfying $\varepsilon_n y_n \rightarrow y_0$, with $V(y_0)=V_0$. Thus, for some $r>0$ it holds $B_r(\varepsilon_n y_n) \subset \Lambda$, and so, $B_{\frac{r}{\varepsilon_n}}(y_n) \subset \Lambda_{\varepsilon_n}$. The last inclusion is equivalent to
			$$
			\mathbb{R}^N-\Lambda_{\varepsilon_n} \subset \mathbb{R}^N-B_{\frac{r}{\varepsilon_n}}(y_n).
			$$	
			On the other hand, the sequence $(y_n)$ can be chosen such that $w_n(x)=u_{\varepsilon_n}(x+y_n)$ satisfies (\ref{39}). Therefore, for $R>0$ large enough,
			$$w_n(x)<t_1,\,\,\,\forall x \in \mathbb{R}^N-B_R(0),$$
			which implies
			$$u_{\varepsilon_n}(x)<t_1,\,\,\, \forall x \in \mathbb{R}^N-B_R(y_n).$$
			Since for $n \in \mathbb{N}$ large enough $r/\varepsilon_n \geq R$, we have
			$$\mathbb{R}^N-\Lambda_{\varepsilon_n} \subset \mathbb{R}^N-B_{\frac{r}{\varepsilon_n}}(y_n)\subset\mathbb{R}^N-B_R(y_n),$$
			for all $n$ large enough, showing that
			$$ u_{\varepsilon_n}(x)<t_1,\,\,\,\forall x \in \mathbb{R}^N-\Lambda_{\varepsilon_n},$$
		which is absurd. This contradiction finishes the proof.
		\end{proof}	
	
	A natural question related with the problem $(P_\varepsilon)$ is about the concentration of positive solutions. Using (\ref{39}), the same arguments employed in \cite[ Section 4]{Alves-Ji} guarantee that the below result holds.
	
	\begin{corollary}[Concentration phenomena] Let $v_\varepsilon(x)=u_{\varepsilon}(x/\varepsilon)$. Then, $v_{\varepsilon}$ is a solution of $(P_\varepsilon)$ for $\varepsilon \in (0,\varepsilon_0)$. Moreover, if $z_\varepsilon \in \mathbb{R}^N$ is a global maximum point of $v_\varepsilon$, we have 
		$$\lim_{\varepsilon\rightarrow 0^+} V(z_\varepsilon)= V_0.$$
	\end{corollary}

	\section{Multiplicity of solution for $(P_\varepsilon)$}
	In this section we will show the existence of multiple solution for $(P_\varepsilon)$ by using the Lusternik-Schnirelmann category theory. More precisely, setting
\begin{equation} \label{M}
M:=\{x \in \Lambda; \,V(x)=V_0\} \,\,\,\text{and}\,\,\,M_\delta:=\{x \in \mathbb{R}^N;\,d(x,M)\leq \delta\},
\end{equation}
where $\delta>0$ is small enough of such way that $M_\delta \subset \Lambda$, our arguments will prove that $(S_\varepsilon)$ has at least $\text{cat}_{M_\delta}(M)$ solutions. To begin with, we start by recalling some notions related with the Lusternik-Schnirelmann category theory, for further details see \cite[Chapter 5, and references therein]{Willem}.
	
	\begin{definition}
		Let $Y$ be a closed subset of a topological space $Z$. We say that the (Lusternik-Schnirelmann) category of $Y$ in $Z$ is $n$, $\text{cat}_{Z}(Y)=n$ for short, if $n$ is the least number of closed and contractible sets in $Z$ which cover $Y$.
	\end{definition}

	Suppose that $W$ is a Banach space and $V$ is a $C^1$- manifold of the form $V=\Psi^{-1}(\{0\})$, where $\Psi \in C^1(W,\mathbb{R})$ and $0$ is a regular value of $\Psi$. For a functional $I:W\longrightarrow\mathbb{R}$ denote
	$$I^d:=\{u \in V;\, I(u)\leq d\}.$$
	
	The following result can be found in \cite[Chapter 5]{Willem} and it is our main abstract tool to get the existence of multiple solution for $(P_\varepsilon)$. 
	
	\begin{theorem}\label{46}
		Let $I \in C^1(W,\mathbb{R})$ be such that $I|_V$ is bounded from below. Suppose that $I$ satisfies the $(PS)_c$ condition for $c \in [\inf I|_V, d]$, then $I|_V$ has at least $\text{cat}_{I^d}(I^d)$ critical points in $I^d$. 
	\end{theorem}
	
	In the sequel, let us introduce some notations that will be used later on. Hereafter, we denote by $u_0$ a  positive ground state solution of $(P_0)$. Furthermore, for each $\delta>0$, we fix $\phi \in C^\infty([0,\infty)$ such that $0\leq\phi\leq 1$ and
	$$\phi(t)=\left\{\begin{aligned}
	&1,\,\,\,&0\leq t\leq \frac{\delta}{2};\\
	&0,\,\,\,&t\geq \delta.
	\end{aligned}
	\right.
	$$
	Using the above notation, for each $y \in M$ we also set 	
	$$
	w_{\varepsilon, y}(x):=\phi(|\varepsilon x - y|)u_0\left(\frac{\varepsilon x-y}{\varepsilon}\right)
	$$
	and let $t_{\varepsilon, y}>0$ be such that $t_{\varepsilon, y}w_{\varepsilon, y} \in \mathcal{N}_\varepsilon.$ Note that $|\text{supp}(w_{\varepsilon, y})\cap\Lambda_\varepsilon|>0$, then we know that $t_{\varepsilon, y}$ verifies
	$J_\varepsilon(t_{\varepsilon, y}w_{\varepsilon, y})=\max_{t \geq 0}J_\varepsilon(tw_{\varepsilon, y}).$
	
	For each $\varepsilon>0$, we define the map
	$$\begin{aligned}
	\Phi_\varepsilon:\,&M\longrightarrow \mathcal{N}_\varepsilon\\
					&y \,\,\,\longmapsto \Phi_{\varepsilon, y}\equiv t_{\varepsilon, y}w_{\varepsilon, y}.
	\end{aligned}
	$$
	Now, fix $\rho>0$ such that $M_\delta \subset B_\rho(0)$ and $\zeta:\mathbb{R}^N \longrightarrow \mathbb{R}^N$ given by 
	$$\zeta(x)=\left\{\begin{aligned}
	&x,\,\,\,&|x| \leq \rho;\\
	&\rho\frac{x}{|x|},\,\,\,&|x|\geq \rho.
	\end{aligned}
	\right.
	$$
	Finally, we set $\beta:\mathcal{N}_\varepsilon\longrightarrow\mathbb{R}^N$ given by
	$$\beta(u):=\frac{\displaystyle{\int_{\mathbb{R}^N}}\zeta(\varepsilon x)|u(x)|^p}{||u||_p^p}.
	$$
	
	\begin{lemma}\label{48}
		The following limit holds
		$$
		\lim_{\varepsilon\rightarrow 0}J_\varepsilon(\Phi_{\varepsilon, y})=c_0, \quad \mbox{	uniformly in} \quad y \in M.
		$$
	\end{lemma}
	\begin{proof}
		Arguing by contradiction, we get sequences $(\varepsilon_n)$ and $(y_n)$, with $\varepsilon_n \rightarrow 0$ and $(y_n) \subset M$, such that
		\begin{equation}\label{43}
		|J_{\varepsilon_n}(\Phi_{\varepsilon_n, y_n})-c_0|\geq\delta_0,
		\end{equation}
		for some $\delta_0>0$. Setting $t_n=t_{\varepsilon_n, y_n}$ and using that $\Phi_{\varepsilon_n, y_n} \in \mathcal{N}_{\varepsilon_n}$, we find
		\begin{equation}\label{42}
		\begin{aligned}
		J_{\varepsilon_n}(\Phi_{\varepsilon_n, y_n})=\frac{t_n^2}{2}&\int_{\mathbb{R}^N}(|\nabla \phi(\varepsilon_n z) u_0(z) |^2 + (V(\varepsilon_n z+ y_n)+1)|\phi(\varepsilon_n z)u_0(z)|^2)+\\
		&+\int_{\mathbb{R}^N}F_1(t_n\phi(\varepsilon_n z)u_0(z))-\int_{\mathbb{R}^N}G_2(\varepsilon_n z+ y_n, t_n\phi(\varepsilon_n z)u_0(z))
		\end{aligned}
		\end{equation}
		and
		\begin{equation}\label{41}
		\begin{aligned}
		&\quad\quad \quad \quad \,\,\, t_n^2\int_{\mathbb{R}^N}(|\nabla \phi(\varepsilon_n z) u_0(z) |^2 + (V(\varepsilon_n x+\varepsilon_n y_n)+1)|\phi(\varepsilon_n z)u_0(z)|^2)=\\
		&=\int_{\mathbb{R}^N}G'_2(\varepsilon_n z+ y_n, t_n\phi(\varepsilon_n z)u_0(z))t_n\phi(\varepsilon_n z)u_0(z)-\int_{\mathbb{R}^N}F'_1(t_n\phi(\varepsilon_n z)u_0(z))t_n\phi(\varepsilon_n z)u_0(z).
		\end{aligned}
		\end{equation}
		Note that, if $z\in B_{\frac{\delta}{\varepsilon_n}}(0)$, then $\varepsilon_n z+y_n\in B_\delta(y_n)\subset M_\delta.$ By (\ref{M}),  we derive that $\varepsilon_n z+y_n \in \Lambda$. Hence, for $z \in B_{\frac{\delta}{\varepsilon_n}}(0)$ one has $G'_2\equiv F'_2$. This information together with (\ref{41}) yields
		$$\begin{aligned}
		&\quad\quad\int_{\mathbb{R}^N}(|\nabla \phi(\varepsilon_n z) u_0(z) |^2 + (V(\varepsilon_n x+y_n)+1)|\phi(\varepsilon_n z)u_0(z)|^2)=\\
		&\quad \quad \quad\quad \,\,\,=\int_{\mathbb{R}^N}|\phi(\varepsilon_n z)u_0(z)|^2\log(|t_n\phi(\varepsilon_n z)u_0(z)|^2)=\\
		&=\int_{\mathbb{R}^N}|\phi(\varepsilon_n z)u_0(z)|^2\log(|\phi(\varepsilon_n z)u_0(z)|^2)+\log(|t_n|^2)\int_{\mathbb{R}^N}|\phi(\varepsilon_n z)u_0(z)|^2.
		\end{aligned}
		$$
		Our next step is proving that, going to a subsequence, $t_n \rightarrow 1$. Since $y_n \in M$, we can assume $y_n \rightarrow y_0\in M$. In this way, the above equality ensures that $(t_n)$ is a bounded sequence. Otherwise, going to a subsequence if necessary, we would have $t_n \rightarrow \infty$ and thus $\log (|t_n|^2)\rightarrow \infty$. Gathering this information with the Lebesgue Dominated Convergence Theorem in the above equality we arrive at a contradiction. 
		
		We may assume that $t_n \rightarrow t_0 \geq 0$. Using the same ideas of preceding paragraph, one can see that $t_0>0$. Finally, by combining the Lebesgue's Theorem with the last equality we find
		$$t_0^2\int_{\mathbb{R}^N}(|\nabla u_0|^2 + V_0|u_0|^2)=\int_{\mathbb{R}^N}|t_0u_0|^2\log(t_0|u_0|^2),$$
		which shows that $t_0=1$, because $u_0$ is a ground state solution of $(P_0)$. As $t_n\rightarrow 1$, the sentence in (\ref{42}) implies that $J_{\varepsilon_n}(\Phi_{\varepsilon_n, y_n})\rightarrow J_0(u_0)=c_0$, contradicting (\ref{43}). The proof is now complete.
	\end{proof}	
	Let us introduce the following set
	$$
	\tilde{\mathcal{N}}_\varepsilon:\{u \in \mathcal{N}_\varepsilon;\,J_\varepsilon(u)\leq c_0+o_1(\varepsilon)\}.
	$$
	Note that the last lemma assures that $\Phi_{\varepsilon, y} \in \tilde{\mathcal{N}}_\varepsilon$.
	
	\begin{lemma}
		The map $\beta$ satisfies 
		$$\lim_{\varepsilon\rightarrow 0}\beta(\Phi_{\varepsilon, y})=y, \quad 	\mbox{uniformly in} \quad y\in M. $$
	
	\end{lemma}
	\begin{proof}
		The idea is the same found in \cite[Lemma 4.2]{Alves-Giovany}. If the result is false, there are sequences $\varepsilon_n \rightarrow 0$ and $(y_n) \subset M$ such that
		$$|\beta(\Phi_{\varepsilon_n, y_n})-y_n|\geq \delta_1,$$
		for some $\delta_1>0$. By using the definition of $\beta$ and setting $z=\frac{\varepsilon_n x-y}{\varepsilon_n}$, we find
		$$\beta(\Phi_{\varepsilon_n, y_n})=y_n+\frac{\displaystyle{\int_{\mathbb{R}^N}}(\zeta(\varepsilon z+y_n)-y_n)|\phi(|\varepsilon_n z|)u_0(z)|^p}{\displaystyle{\int_{\mathbb{R}^N}}|\phi(|\varepsilon_n z|)u_0(z)|^p}.$$
		Without loss of generality, we may assume that $y_n\rightarrow y_0 \in M \subset B_\rho(0)$. Thus, the definition of $\zeta$ together with the Lebesgue Dominated Convergence Theorem implies that
		$$|\beta(\Phi_{\varepsilon_n, y_n})-y_n|=o_n(1),$$
		which is absurd. 
	\end{proof}	
	
	In the next lemma we prove a version of result of Cingolani-Lazzo in \cite[Claim 4.2]{Cingolani-Lazzo}. In that paper the authors have considered a homogenous type nonlinearity while in our case we are working with a logarithmic nonlinearity.
	\begin{lemma}\label{52}
		Let $u_n \in \mathcal{N}_{\varepsilon_n}$. Suppose that $J_{\varepsilon_n}(u_n)\rightarrow c_0$, where $\varepsilon_n \rightarrow 0$. Then, there exists a sequence $(y_n)$ in $\mathbb{R}^N$ such that $w_n(x):=u_n(x+y_n)$ has a convergent subsequence in $X$. Furthermore,
		$$\lim_{n \to +\infty} (\varepsilon_n y_n) = y_0,$$
		for some $y_0 \in M$.
	\end{lemma}
	\begin{proof}
		As made in the proof of Lemma \ref{38}, we have that $\displaystyle{\sup_{n\in \mathbb{N}}}||u_n||_{\varepsilon_n}<\infty$, and so, $(u_n)$ is a bounded sequence in $X$. By Lemmas \ref{29}-$ii)$ and \ref{43}, we know that $c_{\varepsilon_n}=\displaystyle{\inf_{u\in\mathcal{N}_{\varepsilon_n}}}J_{\varepsilon_n}(u)$ and $J_{\varepsilon_n}(u_n)=c_{\varepsilon_n}+o_n(1)$. Therefore, by a slight variant of Ekeland's Variational Principle, there is $v_n \in \mathcal{N}_{\varepsilon_n}$ such that\\ \\
		$i)$ $J_{\varepsilon_n}(v_n)=c_{\varepsilon_n}+o_n(1)$;\\
		$ii)$ $||v_n -u_n||_{\varepsilon_n}\leq o_n(1)$;\\
		$iii)$ $||J'_{\varepsilon_n}(v_n)||_*=o_n(1)$.\\
		
	The reasoning employed in the proof of the Proposition \ref{45} shows that $||J'_{\varepsilon_n}(v_n)||_{X'_{\varepsilon_n}}\rightarrow 0$, where $X'_{\varepsilon_n}$ designates the topological dual space of $X_{\varepsilon_n}$.  From the condition $ii)$ above, 
		$$ J'_{\varepsilon_n}(v_n)v_n=o_n(1).$$
		Now, by following the steps in the proof of Lemma \ref{38}, we get a sequence $(y_n) \subset \mathbb{R}^N$ such that
		$$\lim_{n \to +\infty} (\varepsilon_n y_n) =y_0,$$
		for some $y_0 \in M$. Moreover, the sequence $\tilde{w}_n=v_n(\cdot+y_n)$ has a convergent subsequence in $X$ and thus, using $ii)$ above, $w_n:=u_n(\cdot+y_n)$ has a convergent subsequence in $X$. This finishes the proof.
	\end{proof}	
	The below result relates the number of solutions of $(\tilde{S}_\varepsilon)$ with $\text{cat}_{M_\delta}(M)$.
	\begin{proposition} \label{T1}
		Assume that $(V_1)-(V_2)$ hold and that $\delta$ is small enough. Then, problem $(\tilde{S}_\varepsilon)$ has at least $\text{cat}_{M_\delta}(M)$ solutions, with $\varepsilon \in (0,\varepsilon_1)$, for some $\varepsilon_1>0$.
	\end{proposition}
	\begin{proof}	In this proof we will employ the Theorem \ref{46} with $I=J_\varepsilon$, $V=\mathcal{N}_\varepsilon$ and $d=c_o+o_1(\varepsilon)$. In this case, we have $J_\varepsilon^d=\tilde{\mathcal{N}}_\varepsilon$. On account of Proposition \ref{45}, the functional $J_\varepsilon|_{\mathcal{N}_\varepsilon}$ verifies the $(PS)$ condition, and so, the Theorem \ref{46} guarantees that $J_\varepsilon|_{\mathcal{N}_\varepsilon}$ has at least $\text{cat}_{\tilde{\mathcal{N}}_\varepsilon}(\tilde{\mathcal{N}}_\varepsilon)$ critical points in $\tilde{\mathcal{N}}_\varepsilon=J_\varepsilon^d$. Thereby, by Proposition \ref{47}, $J_\varepsilon$ has $\text{cat}_{\tilde{\mathcal{N}}_\varepsilon}(\tilde{\mathcal{N}}_\varepsilon)$ critical points, from where it follows that $(\tilde{P}_\varepsilon)$ has at least $\text{cat}_{\tilde{\mathcal{N}}_\varepsilon}(\tilde{\mathcal{N}}_\varepsilon)$ solutions.
		
		In order to finish the proof, we will prove
		$$\text{cat}_{\tilde{\mathcal{N}}_\varepsilon}(\tilde{\mathcal{N}}_\varepsilon)\geq\text{cat}_{M_\delta}(M).$$
		Our argument follows the ideas of \cite[Section 6]{Cingolani-Lazzo}. It suffices to consider the case $\text{cat}_{\tilde{\mathcal{N}}_\varepsilon}(\tilde{\mathcal{N}}_\varepsilon)<\infty$. Let $n=\text{cat}_{\tilde{\mathcal{N}}_\varepsilon}(\tilde{\mathcal{N}}_\varepsilon)$ and take $A_1,...A_n$ closed and contractible sets in $\tilde{\mathcal{N}}_\varepsilon$ satisfying $\tilde{\mathcal{N}}_\varepsilon=\displaystyle{\bigcup_{i=1}^{n}} A_i$. In this way, it is possible to find $h_i \in C([0,1]\times A_i,\tilde{\mathcal{N}}_\varepsilon)$, with $h_i(0,u)=u$ and $h_i(1,u)=h_i(1,v_0^i)$, for some fixed $v_0^i \in A_i$, $i \in \{1,...,n\}$. Note that, by Lemma \ref{48}, we have $\Phi_\varepsilon(M)\subset \tilde{\mathcal{N}_\varepsilon}$ for $\varepsilon \approx 0^+$. Also, the map
		$$\beta \circ \Phi_\varepsilon:M \longrightarrow M_\delta$$
		is well defined for $\varepsilon\approx 0^+$. Set
		$$ \begin{aligned}
		\eta:[0,1&]\times M\longrightarrow M_\delta\\
		&(t,y) \,\,\,\longmapsto \eta(t,y)= t\beta(\Phi_{\varepsilon,y})+(1-t)y.
		\end{aligned}
		$$
		By using the properties related with $\beta$, one can see that $\eta$ is well defined and $\beta \circ \Phi_\varepsilon$ is homotopic to inclusion map $i: M \longrightarrow M_\delta$. Since $\Phi_\varepsilon$ is a continuous map, the sets $B_i:=\Phi_\varepsilon^{-1}(A_i)$ are closed subsets of $M$. In addition,
		\begin{equation}\label{49}
		M=\bigcup_{i=1}^{n} B_i.
		\end{equation}
		
		Now we are able to show that $n\geq \text{cat}_{M_\delta}(M)$. Indeed, it remains to prove that, for each $i \in \{1,...,n\}$, the set $B_i$ is contractible in $M_\delta$. To this aim, let
		$$ H_i:[0,1]\times B_i \longrightarrow M_\delta$$
		be given by
		$$H_i(t,u)=\left\{\begin{aligned}
		&\eta(2t,u),\,\,\,\, 0\leq t\leq \frac{1}{2};\\
		&g_i(2t-1),\,\,\,\, \frac{1}{2}\leq t\leq 1, 
		\end{aligned}
		\right.
		$$
		with $g_i(t,u):=\beta(h_i(t,\Phi_{\varepsilon,y}))$. The above conditions on $\eta$ and $h_i$ ensure that $H_i$ is well defined. Furthermore, 
		$$ H_i(0,y)=\eta(0,y)=y\,\,\,\,\text{and}\,\,\,\,H_i(1,y)=\beta(h_i(1,v_0^i)),\,\forall y \in B_i,$$
		which shows that $B_i$ is contractible in $M_\delta$. From (\ref{49}) we get the desired inequality. \end{proof}

	The result below points out an important property of the solutions of $(\tilde{S}_\varepsilon)$ obtained in the last theorem.
	
	\begin{proposition}[Positive solutions counting]
		There exists $\varepsilon_2>0$ such that, for $\varepsilon \in (0,\varepsilon_2)$, it holds
		\begin{itemize}
			\item[$i)$] $(\tilde{S}_\varepsilon)$ has at least $\frac{\text{cat}_{M_\delta}(M)}{2}$ positive solutions, if $\text{cat}_{M_\delta}(M)$ is an even number;
			\item[$ii)$] $(\tilde{S}_\varepsilon)$ has at least $\frac{\text{cat}_{M_\delta}(M)+1}{2}$ positive solutions, if $\text{cat}_{M_\delta}(M)$ is an odd number.  
		\end{itemize}	 
	\end{proposition}
	\begin{proof}
		Take $\varepsilon_2\approx 0^+$ and fix $\varepsilon \in(0,\varepsilon_2)$. If $v_\varepsilon$ is a critical point of $J_\varepsilon(v_\varepsilon)\leq c_0+o_\varepsilon(1)$, we must have $v_\varepsilon^+=0$ or $v_\varepsilon^-=0$. Otherwise, we would have $v_\varepsilon^+$, $v_\varepsilon^-\in \mathcal{N}_\varepsilon$, and so,
		$$
		2c_\varepsilon \leq J_\varepsilon(v_\varepsilon^+)+J_\varepsilon(v_\varepsilon^-)=J_\varepsilon(v_\varepsilon) \leq c_0+o_\varepsilon(1),
		$$
		which is a contradiction for  $\varepsilon_2\approx 0^+$. Therefore, using the same arguments of Lemma \ref{54}, we deduce that either $v_\varepsilon>0$ or $v_\varepsilon<0$. 
		
		Now, suppose that $k:=\text{cat}_{M_\delta}(M)$ is an even number and let $v_1,...,v_k$ be the solutions of $(\tilde{P}_\varepsilon)$ given in the preceding proposition. If at least $\frac{k}{2}$ of the solutions $v_1,...,v_k$ are positive solutions, the item $i)$ is proved. Otherwise, we know that at least $\frac{k}{2}$ of the solutions $v_1,...,v_k$ are negative. Denote by $w_1,...,w_{\frac{k}{2}}
		$ such negative solutions. Since $g_2(x,\cdot)-F'_1$ is an odd function, the functions $-w_1,...,-w_{\frac{k}{2}}$ are positive solutions of the problem
		\begin{equation}
		\left\{\begin{aligned}
		-&\Delta u + (V(\varepsilon x)+1)u =g_2(\varepsilon x,u)-F'_1(u),\;\;\mbox{in}\;\;\mathbb{R}^{N},\nonumber \\
		&u \in H^{1}(\mathbb{R}^{N})\cap L^{F_1}(\mathbb{R}^N).
		\end{aligned}
		\right.\leqno{(\tilde{S}_\varepsilon)}
		\end{equation}
		and thus $i)$ is proved. The proof of $ii)$ follows by a similar reasoning.
		\end{proof}

\subsection{Proof Theorem \ref{TherFinal}}
Let $v_\varepsilon$ be a critical point of $J_\varepsilon(v_\varepsilon)\leq c_0+o_\varepsilon(1)$. It suffices to show that there exists $\varepsilon_3\approx 0^+$ such that, for $\varepsilon \in (0,\varepsilon_3)$,
		\begin{equation}\label{51}
		0<v_\varepsilon(x)< t_1,\,\,\,\forall x \in \mathbb{R}^N-\Lambda_\varepsilon,
		\end{equation}
		for each solution $v_\varepsilon$ of $(\tilde{S}_\varepsilon)$ given in the items $i)-ii)$ of the last proposition. Arguing by contradiction, we get a sequence  $(v_{\varepsilon_n})$ of solutions of $(\tilde{S}_{\varepsilon_n})$ where $\varepsilon_n\rightarrow 0$ and $v_n:=v_{\varepsilon_n}$ does not  satisfy (\ref{51}). Note that the obtained sequence $(v_n)$ satisfies the hypothesis of Lemma \ref{52} and that the sequence $(w_n)$ given in the lemma must satisfy (\ref{39}). Thus, a contradiction is obtained by following closely the same ideas used in the proof of Theorem \ref{50}. This argument ensures that $(S_\varepsilon)$ verifies $i)-ii)$ in the statement of the Theorem \ref{TherFinal}. Now, the result follows by a change of variable. \;\;\;\;\;\;\;\;\;\;\;\;\;\;\;\;\; \hspace{11.5cm}$\Box$
	
\section{Final comments}
			In \cite{Alves-Giovany,Cingolani-Lazzo} the result of multiplicity of solution involving the Lusternik-Schnirelmann category assures the existence of at least $\text{cat}_{M_\delta}(M)$ positive solutions. In \cite{Alves-Giovany}, for example, the key point  is the fact that the nonlinearity $f$ was assumed such that $f(t)=0$, $t\leq 0$. In our case, this framework lead us to consider $f(t)=|t^+|^2\log |t^+|^2$, as well as,
			$$J_\varepsilon(u):=\frac{1}{2}\int_{\mathbb{R}^{N}}(|\nabla u|^2+(V(\varepsilon x)+1)|u|^2)+\int_{\mathbb{R}^N}F_1(u^+)-\int_{\mathbb{R}^{N}}G_2(\varepsilon x,u^+),\,\,\,\forall u \in X_\varepsilon.$$
			However, we were not able to reproduce some estimates made throughout this work by considering $J_\varepsilon$ given as above. For example, in the Lemma \ref{9}, we were not able to show the boundedness of the $(PS)$ sequences when $J_\varepsilon$ is chosen in this way. In fact, since the norm on $X_\varepsilon$ involves the norm $||\cdot||_{F_1}$ of Orlicz space $L^{F_1}(\mathbb{R}^N)$, we need of the information of term $\displaystyle{\int_{\mathbb{R}^N}}F_1(u)$ in our computations. This justifies 
			because our number of positive solutions by using the Lusternik-Schnirelmann category is a little bit different from that given in \cite{Alves-Giovany,Cingolani-Lazzo} .\\
\vspace{0.2 cm}

\noindent {\bf Acknowledgments:} The authors would like to thank the anonymous referees by their valuable and important suggestions, which were very important to improve the paper.

	
\end{document}